\documentclass[]{amsart}
\usepackage[utf8]{inputenc}
\usepackage{amsmath}
\usepackage{amssymb}
\usepackage{amsthm}
\usepackage{tikz}
\usepackage{xcolor}
\usepackage{hyperref}
\usepackage{url}

\hypersetup{
	colorlinks=true,       
	linkcolor=cyan,          
	citecolor=green,        
	filecolor=magenta,      
	urlcolor=cyan           
}

\numberwithin{equation}{section}
\numberwithin{figure}{section}

\allowdisplaybreaks
\theoremstyle{plain}
\newtheorem{lemma}{Lemma}[section]
\newtheorem{proposition}[lemma]{Proposition}
\newtheorem{corollary}[lemma]{Corollary}

\newtheorem{theorem}[lemma]{Theorem}

\theoremstyle{definition}
\newtheorem{definition}[lemma]{Definition}
\newtheorem{remark}[lemma]{Remark}

\newtheorem{example}[lemma]{Example}

\newcommand{\R}{\mathbb{R}}
\newcommand{\Z}{\mathbb{Z}}
\newcommand{\N}{\mathbb{N}}
\newcommand{\Q}{\mathbb{Q}}
\newcommand{\sone}{\mathbb{S}^1}

\newcommand{\ind}{\mathbf{1}}

\renewcommand{\P}{\mathbb{P}}
\newcommand{\E}{\mathbb{E}}

\newcommand{\F}{\mathcal{F}}
\newcommand{\supp}{\mathrm{supp}}

\newcommand{\esssup}{\text{ess sup}}

\newcommand{\weak}{\tau_{w}}
\newcommand{\weakp}{\tau_{w}^p}

\newcommand{\closed}{\mathrm{C}}
\newcommand{\cpt}{\mathrm{K}}

\newcommand{\kurouter}{\mathrm{Ls}}
\newcommand{\kurinner}{\mathrm{Li}}
\newcommand{\kurlimit}{\mathrm{Lt}}
\newcommand{\upkuratowski}{\tau_{K}^{+}}
\newcommand{\downkuratowski}{\tau_{K}^{-}}

\newcommand{\Fell}{\tau_{F}}

\newcommand{\uphausdorfftopo}{\tau_{H}^{+}}
\newcommand{\downhausdorfftopo}{\tau_{H}^{-}}
\newcommand{\hausdorfftopo}{\tau_{H}}

\newcommand{\effros}{\mathcal{E}}
\newcommand{\pointwise}{\tau_{\text{PW}}}

\begin{document}
	
	\title{Limit Theorems for Fr\'echet Mean Sets}
	
	\author{Steven N. Evans}
	\address{Department of Statistics\\
		University  of California\\ 
		367 Evans Hall \#3860\\
		Berkeley, CA 94720-3860 \\
		U.S.A.}
	\email{evans@stat.berkeley.edu}
	
	\author{Adam Q. Jaffe}
	\address{Department of Statistics\\
		University  of California\\ 
		367 Evans Hall \#3860\\
		Berkeley, CA 94720-3860 \\
		U.S.A.}
	\email{aqjaffe@berkeley.edu}
	\thanks{This material is based upon work for which AQJ was supported by the National Science Foundation Graduate Research Fellowship under Grant No. DGE 1752814.}
	
	\subjclass[2010]{60F15, 51F99, 54C60}
	
	\keywords{Kuratowski convergence, Hausdorff metric, Wasserstein metric, random sets, Karcher mean, non-Euclidean statistics, medoids}
	
	\date{\today}
	
	\begin{abstract}
		For $1\le p \le \infty$, the Fr\'echet $p$-mean of a probability measure on a metric space is an important notion of central tendency that generalizes the usual notions in the real line of mean ($p=2$) and median ($p=1$).
		In this work we prove a collection of limit theorems for Fr\'echet means and related objects, which, in general, constitute a sequence of random closed sets.
		On the one hand, we show that many limit theorems (a strong law of large numbers, an ergodic theorem, and a large deviations principle) can be simply descended from analogous theorems on the space of probability measures via purely topological considerations.
		On the other hand, we provide the first sufficient conditions for the strong law of large numbers to hold in a $T_2$ topology (in particular, the Fell topology), and we show that this condition is necessary in some special cases.
		We also discuss statistical and computational implications of the results herein.
	\end{abstract}
	
	\maketitle
	

	\section{Introduction}\label{sec:intro}
	For $0 \le r < \infty$ and a metric space $(X,d)$, denote by $\mathcal{P}_r(X)$ the set of Borel probability measures $\mu$ on $X$ such that $\int_X d^r(x,y) \, d\mu(y) < \infty$ for all $x \in X$.
	For $1\le p < \infty$, one defines the \textit{Fr\'echet $p$-mean} of a probability measure $\mu\in\mathcal{P}_p(X)$ to be the set
	\begin{equation}\label{eqn:Frechet-mean}
		F_p(\mu) := \underset{x\in X}{\arg\,\min}\int_{X}d^p(x,y)\, d\mu(y).
	\end{equation}
	By the \textit{empirical Fr\'echet $p$-mean} of finitely many points $y_1,\dots, y_n\in X$ we mean the Fr\'echet $p$-mean of their empirical measure, that is
	\begin{equation}\label{eqn:emp-Frechet-mean}
		F_p\left(\frac{1}{n}\sum_{i=1}^{n}\delta_{y_i}\right) = \underset{x\in X}{\arg\,\min}\sum_{i=1}^{n}d^p(x,y_i).
	\end{equation}
	We take these sets to be empty if no minimizer exists.
	In fact, we will later see that Fr\'echet $p$-means are actually well-defined via a slightly more involved definition
	for $\mu\in\mathcal{P}_{p-1}(X)$.
	
	We observe in Lemma~\ref{lem:Fp-closed-bdd} that the Fr\'echet $p$-mean is a closed set.
	It may have more than one point (for example, if $X = \R$ with the usual metric and $\mu= \frac{1}{2}\delta_0+\frac{1}{2}\delta_{1}$, then $F_1(\mu) = [0,1]$), and it also may be empty (for example, if $X = \R \setminus \{\frac{1}{2}\}$ equipped with the metric inherited from the usual metric on $\R$ and $\mu=\frac{1}{2}\delta_0+\frac{1}{2}\delta_{1}$, then $F_2(\mu) = \emptyset$).

	If $X$ is an inner product space and $d$ is the metric induced by its norm, then the Fr\'echet $2$-mean of the uniform measure on finitely many points coincides with their arithmetic mean; more generally, if $X$ is a separable Hilbert space, then the Fr\'echet $2$-mean of a Borel probability measure $\mu$ coincides with the Bochner integral $\int_{X}x\, d\mu(x)$.
	Moreover, for $X = \R^k$ with the usual metric, the Fr\'echet $1$-mean is exactly the geometric median, and, for $k=1$, the Fr\'echet $1$-mean coincides with the median.
	
	More broadly, the appeal of Fr\'echet means is that they generalize 
	common notions of central tendency in spaces where the structure of addition need not exist.
	For this reason, Fr\'echet $p$-means have become an important tool for the theory and practice of probability and statistics on non-Euclidean spaces \cite{ Chakraborty_2015_ICCV,StiefelStatistics,dryden2009non, fiori2009learning,Huckemann2015, le2000frechet,turner2014frechet}, especially in the recent the statistical theory of graph-valued random variables \cite{GinestetGraphs,UnlabeledGraphs}.
	The Fr\'echet 2-mean is by far the oldest and most important instance of the general theory, and is sometimes referred to directly as {\bf the} Fr\'echet mean or the \textit{Karcher mean}.
	
	Fr\'echet $p$-means also have a nice interpretation in the context of optimal transport:
	If $W_p$ represents the $p$-Wasserstein metric on $\mathcal{P}_p(X)$, then, since points in the Fr\'echet $p$-mean $F_p(\mu)$ are exactly minimizers of the map $x\mapsto W_p(\mu,\delta_x)$, they can be regarded as the best one-point approximations of $\mu$, as quantified by $W_p$.
	This perspective leads us to some interesting geometric insights about the structure of Fr\'echet $p$-means.
	
	In addition to their utility in statistics, Fr\'echet 2-means are also of great theoretical interest.
	Classically, they arose naturally in differential geometry, where they were used in the proof of Cartan's fixed point theorem for groups of isometries acting on a compact manifold of non-positive curvature \cite{Cartan}, and in understanding the algebraic structure of ``nearby'' actions of compact Lie groups on connected compact differentiable manifolds \cite{GroveKarcher}. More recently, they have been successfully used in optimal-transport-based metric thickening problems arising in algebraic topology \cite{MetricThickenings}.
	
	As Fr\'echet $p$-means purport to be a sort of average, a natural question is whether they satisfy generalizations of the limit theorems that hold in Euclidean spaces for averages of random sequences.
	The most basic of these is the law of large numbers and hence we are lead to ask:
	If $Y_1,Y_2,\ldots$ is an independent, identically distributed (i.i.d.) sequence of $X$-valued random elements with common distribution $\mu$, and if $\bar \mu_n$ is the empirical distribution of $Y_1, \ldots, Y_n$, does it follow that the empirical Fr\'echet $p$-means $F_p(\bar \mu_n)$ almost surely converge in some sense to the population Fr\'echet $p$-mean $F_p(\mu)$?
	
	A variant of the Fr\'echet means is also an important object of study in machine learning.
	In clustering tasks, when attempting to summarize a given collection of data via a single datum, it can desirable for various reasons that this datum be chosen, not among all points in the abstract space in which the data live, but rather among the existing data points; in these cases one is often interested in computing ``medoids'' \cite{LinearMedoids,UltraFastMedoids,Kaufman}.
	Similar constructions have been used in non-Euclidean statistics, like in the algorithms for computing prinicple nested spheres studied in \cite{Huckemann}.
	
	A mathematically precise definition of this procedure is as follows.
	For $1 \le p < \infty$, the \textit{support-restricted Fr\'echet $p$-mean} of a probability measure $\mu\in\mathcal{P}_p(X)$ is the set
	\begin{equation}\label{eqn:restr-Frechet-mean}
		F_p^\ast(\mu) := \underset{x\in \supp(\mu)}{\arg\,\min}\int_{X}d^p(x,y)\, d\mu(y),
	\end{equation}
	where $\supp(\mu):=\{x\in X: \text{ if } x\in V \subseteq X \text{ and $V$ is open, then } \mu(V) > 0\}$ is the (topological) support of $\mu$.
	By the \textit{empirical support-restricted Fr\'echet $p$-mean} of finitely many points $y_1,\dots y_n\in X$, we mean the support-restricted Fr\'echet $p$-mean of their empirical measure, that is
	\begin{equation}
		F_p^\ast\left(\frac{1}{n}\sum_{i=1}^{n}\delta_{y_i}\right) = \underset{y_j: j = 1,\ldots n}{\arg\,\min}\sum_{i=1}^{n}d^p(y_j,y_i).
	\end{equation}
	Some authors refer to empirical support-restricted Fr\'echet $p$-means as \textit{medoids} but other authors reserve the term for the case of $p=1$.
	As before, we take these sets by convention to be empty if no minimizer exists, and we will see that we can also suitably extend this definition to all $\mu\in\mathcal{P}_{p-1}(X)$.
	
	Following this motivation, another focus of this work is the question:
	If $Y_1,Y_2,\ldots$ is an i.i.d. sequence of $X$-valued random variables with common distribution $\mu$, and if $\bar \mu_n$ is the empirical distribution of $Y_1,\dots Y_n$, do the empirical support-restricted Fr\'echet $p$-means $F_p^\ast(\bar \mu_n)$ converge almost surely in some sense to the population support-restricted Fr\'echet $p$-mean $F_p^\ast(\mu)$?
	
	It will be useful for our later work to unify both Fr\'echet $p$-means and support-restricted Fr\'echet $p$-means under the same framework.
	The natural way to do this is as follows.
	For $1\le p < \infty$, $\mu\in \mathcal{P}_{p}(X)$, and a closed set $C\subseteq X$, define the \textit{$C$-restricted Fr\'echet $p$-mean} to be the set
	\begin{equation}\label{eqn:ext-Frechet-mean}
		F_p(\mu,C) := \underset{x\in C}{\arg\,\min}\int_{X}d^p(x,y)\, d\mu(y).
	\end{equation}
	As always, this set can be non-empty if no minimizers exist, and the notion can be suitably defined under the weaker moment assumption $\mu\in \mathcal{P}_{p-1}(X)$.
	In particular, note that $F_p(\mu,X)$ is exactly the Fr\'echet $p$-mean and that $F_p(\mu,\supp(\mu))$ is exactly the support-restricted Fr\'echet $p$-mean.
	In this language, our goal can simply be stated as understanding limit theorems for restricted Fr\'echet means, where we ``restrict'' to general sequences of random sets, including the trivial case of no restriction.
	
	The relationship between our restricted Fr\'echet means and the so-called \textit{generalized Fr\'echet means} of \cite{Huckemann2015} is subtle.
	On the one hand, our notion of the $C$-restricted Fr\'echet mean corresponds to taking (in the notation and language of \cite{Huckemann2015}) the data space as $Q:=X$ and the descriptor space as $P:=C$.
	On the other hand, this perspective slightly misses the point, since we need to allow $C$ to be a random set and we need to allow it to vary with the number of data $n$.
	For example, limit theorems for support-restricted Fr\'echet means require one to investigate the convergence $F_p(\bar \mu_n,\supp(\bar \mu_n))\to F_p(\mu,\supp(\mu))$;
	from the perspective of generalized Fr\'echet means, one would need to take $P_n := \supp(\bar \mu_n)$, but existing results are unable to say much about this matter.
	
	Even in the case of (unrestricted) Fr\'echet means, a primary technical difficulty that arises is that Fr\'echet means are typically sets with more than one point, so it is not clear how to state (let alone, how to prove) the limit theorems of interest.
	Many authors \cite{Bhattacharya,FrechetCLTs,ManifoldsI, ManifoldsII,BookCLTs} prefer to focus on a version of the problem which avoids this issue:
	They assume that the population Fr\'echet mean is a single point, and they consider a measurable selection of the empirical Fr\'echet mean.
	In this way, many results parallel their classical counterparts quite closely, and many strong conclusions, including central limit theorems \cite{FrechetCLTs,BookCLTs}, can be shown to hold.
	A particularly well-developed area of this theory is that in which $(X,d)$ arises from a Riemannian manifold; the monograph \cite{ManifoldInference} contains a comprehensive account of what is known in this realm.
	
	In the present work we embrace the generality that the restricted Fr\'echet means are random sets.
	This perspective was taken up by a few previous authors \cite{Ziezold,Sverdrup,GinestetGraphs,Huckemann} who found a meaningful sense in which the strong law of large numbers (SLLN) can be shown to hold:
	If $C_1,C_2,\ldots$ is a sequence of closed subsets of a metric space $(X,d)$, we define their \textit{Kuratowski upper limit} to be the closed set
	\[
	\underset{n\to\infty}{\kurouter}C_n =\bigcap_{n=1}^{\infty}\overline{\bigcup_{m=n}^{\infty}C_m}.
	\]
	Then, under some rather limiting assumptions, it was shown that empirical Fr\'echet means of IID samples converge to the population Fr\'echet means in the sense that we have $\kurouter_{n\to\infty}F_p(\bar \mu_n)\subseteq F_p(\mu)$ almost surely \cite{Ziezold,Sverdrup,Huckemann}; under similar assumptions it was also shown that empirical support-restricted Fr\'echet means converge to population support-restricted Fr\'echet means in the sense that we have $\kurouter_{n\to\infty}F_p^{\ast}(\bar \mu_n)\subseteq F_p^{\ast}(\mu)$ almost surely \cite{Sverdrup}.
	This notion of convergence is sometimes referred to as \textit{Ziezold consistency (ZC)} because of the seminal work \cite{Ziezold}.
	Note in both cases that we have inclusions rather than exact equality.
	
	Since almost sure containment of the Kuratowski upper limit is quite a weak notion of convergence, a few other authors have considered slightly stronger notions \cite{ManifoldsI,ManifoldsII}.
	That is, for closed bounded sets $C,C'\subseteq X$, we define their \textit{one-sided Hausdorff distance} to be
	\begin{equation*}
		\rho(C,C'):=\max_{x\in C}\min_{x'\in C'}d(x,x').
	\end{equation*}
	We note that $\rho$ is not a metric, and that we have $\rho(C,C') = 0$ if and only if $C\subseteq C'$.
	Then, under a few further assumptions \cite{ManifoldsI,ManifoldsII} on the metric space and on the population distribution, it has been shown that empirical Fr\'echet means of i.i.d. samples converge to the population Fr\'echet means in the sense that we have $\rho(F_p(\bar \mu_n),F_p(\mu))\to 0$ almost surely.
	This notion of convergence is sometimes called \textit{Bhattacharya-Patrangenaru consistency (BPC)} because of the seminal works \cite{ManifoldsI,ManifoldsII}
	We note that BPC can in some sense be viewed as a ``uniform'' version of ZC, and that BPC also features some form of inclusion rather than exact equality.
	
	Any result which establishes almost sure convergence of the empirical restricted Fr\'echet means of i.i.d. samples to the population restricted Fr\'echet mean, under a particular notion of convergence, will be referred to as a \textit{strong law of large numbers (SLLN)} for restricted Fr\'echet means.
	The results above, in this language, establish SLLNs for the Kuratowski upper convergence and for the one-sided Hausdorff convergence, under suitable assumptions.
	In certain settings, an SLLN may coincide with almost sure convergence in some topology, but, as we will later see, this is not always the case.
	With this level of generality in mind, the goal of this paper is to refine our understanding of these SLLNs, hopefully showing that they hold under minimal assumptions, and to understand when we have SLLNs in other better-understood senses of convergence (in the Fell topology, in the Hausdorff metric, etc.).
	
	As such, one contribution of the present work is to establish SLLNs for restricted Fr\'echet means in cases beyond those that are already covered by the seminal works \cite{Ziezold,Sverdrup,ManifoldsI,ManifoldsII} and by the most recent literature \cite{Huckemann}.
	In particular, for unrestricted Fr\'echet means and support-restricted Fr\'echet means, we establish that, if $(X,d)$ is separable and $\mu\in\mathcal{P}_{p-1}(X)$, then (Theorem~\ref{thm:Frechet-SLLN}) we have almost sure containment of the Kuratowski upper limit, and moreover that, if $(X,d)$ has the Heine-Borel property (meaning that its closed, bounded sets are compact) and $\mu\in\mathcal{P}_{p}(X)$, then (Theorem~\ref{thm:Frechet-SLLN-hausdorff}) we have almost sure convergence in the one-sided Hausdorff distance.
	In all of these settings, we also show that the same results hold even if one enlarges the restricted Fr\'echet mean set to a slightly larger set of \textit{relaxed} restricted Fr\'echet means.
	However, all of these notions of convergence are ``one-sided'', meaning that, while features appearing at the empirical level must also be present at the population level, there may be features at the population level which are not present at the empirical level.
	We also show that this ``one-sidedness'' can be strict (see Examples~\ref{ex:discrete-uniform}, \ref{ex:both-strict}, and \ref{ex:irrationals}), thereby answering an open question of Huckemann from \cite{HuckemannOberwolfach}.
	
	Hence, an additional contribution of this work is to establish some sufficient conditions under which we have ``two-sided'' convergence almost surely.
	In particular, we show (Theorem~\ref{thm:Fp-full-SLLN}) that, if $(X,d)$ has the Heine-Borel property, if $\mu\in\mathcal{P}_{p}(X)$, and if the population Fr\'echet mean is a singleton up to an explicit notion of equivalence that we introduce, then we have almost sure convergence in the Fell topology.
	We show (Theorem~\ref{thm:full-SLLN-finite}) that this sufficient condition is also necessary in the case that $(X,d)$ has finitely many points, but we do not know whether it is necessary in general Heine-Borel spaces.
	As far as we know, this is the first result of its kind in the general set-valued case.
	We also fill in a technical gap in much of the existing literature by clarifying some questions about the measurability of various random sets with respect to the Effros $\sigma$-algebra on the space of all closed subsets of $X$.
	
	Yet, perhaps more important than our concrete results is our method of proof.
	Most existing results in this realm \cite{ManifoldsI,ManifoldsII,GinestetGraphs,Sverdrup,Ziezold} proceed with a general approach which shows that the solution sets to a sequence of suitable stochastic optimization problems have an almost sure limiting solution set.
	While this approach is fruitful, it relies heavily on unnecessarily powerful results like uniform laws of large numbers.
	In contrast, our work develops the same results by directly studying continuity-type properties of the restricted Fr\'echet mean map.
	In particular, we show that the SLLN merely ``descends'' from an analogous strong limit theorem in a space of measures.
	This approach has two significant advantages:
	The first is that it allows for the possibility of descending other limit theorems in the space of measures to limit theorems for Fr\'echet means; for example, we prove an ergodic theorem (Theorem~\ref{thm:Fp-ergodic-MC}) and a large deviations principle (Theorem~\ref{thm:Fp-LDP}) with essentially no additional work.
	The second is that this approach allows for the possibility that the optimization region and the relaxation threshhold are chosen adaptively as a function of the data.
	This flexibility is exploited in two forthcoming works by the second author, one on strong consistency for adaptive clustering procedures which are variants of $k$-means \cite{JaffeClustering}, and one on finer statistical questions about estimting Fr\'echet mean sets \cite{Relaxation}.
	
	While previous research used the aforementioned ideas from stochastic optimization to prove only partial results on the almost sure one-sided convergence of Fr\'echet means, we remark that the concurrent work \cite{Schoetz2} uses the same techniques to prove results similarly powerful to our own.
	Specifically, the main theorems therein show, for (unrestricted) Fr\'echet means, that one gets almost sure containment of the Kuratowski upper limit whenever $(X,d)$ is separable and $\mu\in\mathcal{P}_{p-1}(X)$ and that one gets almost sure convergence in the one-sided version of the Hausdorff metric whenever $(X,d)$ is separable and has the Heine-Borel property and $\mu\in\mathcal{P}_{p-1}(X)$.
	The latter result is stronger than our own in the sense that we require the stricter moment assumption $\mu\in\mathcal{P}_{p}(X)$ for the case of one-sided Hausdorff convergence; we conjecture but are unable to prove that the stronger result should also follow from an analogous continuity-type theorem.
	However, we point out that the results of \cite{Schoetz2} do not cover the case of support-restricted Fr\'echet means, the case of two-sided convergence, nor the extension to other kinds of limit theorems.
	
	Finally, we provide some insight into a computational perspective on Fr\'echet means.
	Since (unrestricted) Fr\'echet means are notoriously difficult to compute exactly in all but a few simple settings, there is a vast literature on algorithmic aspects of computing Fr\'echet means \cite{FrechetCircle,Chakraborty_2015_ICCV,FrechetSphere,KarcherSOn,FrechetDifferentiating} especially in spaces of phylogenetic trees \cite{FrechetHadamard,TreeMean,TreeOptimality}; software packages like \cite{GeomStats} have implemented  these methods for easy use in applied statistics and machine learning problems.
	For support-restricted Fr\'echet  means, the drawbacks are similar; even though the empirical support-restricted Fr\'echet means of $n$ data points can be computed exactly in $O(n^2)$ steps, this is still intractable for large $n$, hence there is great interest in approximating $F_p^{\ast}(\bar \mu_n)$ in (sub)linear time, usually through various randomized algorithms \cite{LinearMedoids,UltraFastMedoids}.
	Our work suggests that, unless it is known that the SLLN holds in some two-sided sense (say, in the Fell topology), then more sophisticated algorithms (and, possibly, more computation) may be needed.
	
	The remainder of this paper is structured as follows.
	In Section~\ref{sec:topology} we detail some notions of convergence of measures and of closed sets which will be used throughout the paper.
	In Section~\ref{sec:measurability} we clarify some measurability properties of restricted Fr\'echet means and other random sets, and Section~\ref{sec:results} contains the statements and proofs of the main probabilistic results.
	
	\section{Topological Preliminaries}\label{sec:topology}
	
	In this section we describe various topologies on suitable spaces of measures and various notions of convergence on suitable spaces of closed subsets of a given metric space, and these are the core of the work in the subsequent sections.
	While most of these notions can be defined in a more general setting where $X$ is just assumed to be a topological space, we are only interested in the setting where $X$ is a metric space, and this will simplify our work slightly.
	From here on we thus always assume that $X$ is a set with metric $d$ unless otherwise stated. 
	
	To begin, we describe a few spaces of measures and various topologies on them.  
	Write $\mathcal{P}(X)$ for the Borel probability measures on $(X,d)$, and write $\supp(\mu)$ for the closed set $\{x\in X: \text{ if } x\in V \subseteq X \text{ open, then } \mu(V) > 0\}$ called the \textit{support} of $\mu\in\mathcal{P}(X)$.
	Define $\mathcal{P}_r(X)$ for $r> 0$ to be the collection of all $\mu\in\mathcal{P}(X)$ satisfying $\int_{X}d^r(x,y) \, d\mu(y) < \infty$ for all $x\in X$, and observe that this coincides with our earlier definition as long as $1\le r < \infty$.
	In this paper we always adopt the convention that $0^0 = 1$ in order to avoid checking the $r=0$ case separately; in particular, we have $\mathcal{P}_0(X) = \mathcal{P}(X)$.
	
	Next we derive an elementary but important inequality.
	If $0 \le r \le 1$, then the subadditivity of the function $t\mapsto t^r$ and the triangle inequality yield
	\begin{equation}
		\label{eqn:distance_comparison_subadditive}
		d^r(x',y) \le d^r(x',x'') + d^r(x'',y), \quad x',x'',y \in X,
	\end{equation}
	and if $1\le r < \infty$, then the convexity of the function $t \mapsto t^r$ and the triangle inequality yield
	\begin{equation}
		\label{eqn:distance_comparison_convex}
		d^r(x',y) \le 2^{r-1}(d^r(x',x'') + d^r(x'',y)), \quad x',x'',y \in X.
	\end{equation}
	Thus for $r \ge 0$ we have
	\begin{equation}
		\label{eqn:distance_comparison}
		d^r(x',y) \le c_r(d^r(x',x'') + d^r(x'',y)), \quad x',x'',y \in X.
	\end{equation}
	where $c_r =\max\{2^{r-1},1\}$.
	Consequently, note that, for $r\ge 0$, a distribution $\mu\in\mathcal{P}(X)$ satisfies $\mu\in \mathcal{P}_{r}(X)$ if and only if we have $\int_{X}d^r(x,y) \, d\mu(y) < \infty$ for {\bf some} $x\in X$.
	
	For $r\ge 0, \mu\in\mathcal{P}_r(X)$, and $x\in X$, we define the value $g_r(\mu,x)$ as
	\begin{equation}
		g_{r}(\mu,x) := \int_{X}d^r(x,y)\, d\mu(y).
	\end{equation}
	In general let us also write $g_r(\mu,x) =: g_{r,\mu}(x) =: g_{r,x}(\mu)$ whenever we wish to emphasize that certain arguments are held fixed while others are allowed to vary.
	
	Let us endow $\mathcal{P}(X)$ with the topology of weak convergence which we denote by $\weak$. 
	Further, for $r\ge 0$, let us say that a sequence $\{\mu_n\}_{n \in \N}$ in $\mathcal{P}_r(X)$ \textit{converges weakly in $\mathcal{P}_r(X)$} to $\mu \in \mathcal{P}_r(X)$ if $\mu_n\to \mu$ in $\weak$ and $g_{r,x}(\mu_n)\to g_{r,x}(\mu)$ as $n\to\infty$ for all $x\in X$; this notion of convergence arises from a metrizable topology $\tau_{w}^{r}$ on $\mathcal{P}_r(X)$.
	
	\begin{lemma}
		\label{lem:one_x_all_x}
		Suppose $r\ge 0$ and that the sequence $\{\mu_n\}_{n \in \N}$ in $\mathcal{P}_r(X)$ converges in $\weak$ to $\mu \in \mathcal{P}_r(X)$.  The following are equivalent:
		\begin{itemize}
			\item[(i)]
			$g_{r,\mu_{n}}(x) \to g_{r,\mu}(x)$ for {\bf all} $x\in X$ (so that $\mu_n \to \mu$ in $\tau_{w}^{r}$);
			\item[(ii)] $g_{r,\mu_{n}}(x) \to g_{r,\mu}(x)$ for {\bf some} $x\in X$.
		\end{itemize}
	\end{lemma}
	
	\begin{proof}
		We need only prove that (ii) implies (i); that is, we need to show for $x, x' \in X$ that if $\int_{X}d^r(x,y)\, d\mu_n(y) \to \int_{X}d^r(x,y)\, d\mu(y)$, then $\int_{X}d^r(x',y)\, d\mu_n(y) \to \int_{X}d^r(x',y)\, d\mu(y)$.
		Suppose that $(Z_n)_{n \in \N}$ and $Z$ are random elements of $X$, where $Z_n$, $n \in \N$, has distribution $\mu_n$ and $Z$ has distribution $\mu$.
		Put $W_n= d^r(x, Z_n)$, $W = d^r(x,Z)$, $W_n'= d^r(x', Z_n)$, and $W' = d^r(x',Z)$.
		By assumption, $Z_n$ converges to $Z$ in distribution.
		Therefore, by the continuous mapping theorem \cite[Theorem~4.27]{Kallenberg}, $W_n$ converges to $W$ and $W_n'$ converges to $W'$ in distribution.
		Also by assumption, $\E[W_n]$ converges to $\E[W]$, so it follows that the sequence $\{W_n\}_{n \in \N}$ is uniformly integrable \cite[Lemma~4.11]{Kallenberg}.
		By \eqref{eqn:distance_comparison}, $W_n' \le c_{p}(d^r(x,x') + W_n)$ and so the sequence $\{W_n'\}_{n \in \N}$ is also uniformly integrable.
		Consequently, $\E[W_n']$ converges to $\E[W']$, again by \cite[Lemma~4.11]{Kallenberg}.
	\end{proof}
	
	The obvious relationships between these spaces and these topologies are summarized in the following result.
	The proof is elementary, and is hence omitted for the sake of brevity.
	
	\begin{lemma}\label{lem:Wasserstein-relations}
		For $0 \le r \le r' < \infty$, there is the inclusion $\mathcal{P}_{r'}(X) \subseteq \mathcal{P}_r(X)$, and the restriction of $\tau_{w}^{r}$ to $\mathcal{P}_{r'}(X)$ is weaker than $\tau_{w}^{r'}$.
		Moreover, if $(X,d)$ is compact, then $\mathcal{P}_r(X) = \mathcal{P}(X)$ and $\tau_{w}^{r} = \weak$ for all $0 \le r < \infty$.
	\end{lemma}

	We also require the following result which is elementary and well-known.
	We don't know where it is proved in the literature, so we provide the simple proof.
	
	\begin{lemma}
		\label{lem:Peter_Paul}
		Given $r\ge0$ and $\varepsilon > 0$, there exists some constant $c_{r,\varepsilon} > 0$ such that $(a+b)^r \le (1+\varepsilon)a^r + c_{r,\varepsilon}b^r$ for all $a,b\ge 0$.  Consequently, for $x',x'',y \in X$ there is the inequality $d^r(x',y) \le (1+\varepsilon) d^r(x'',y) + c_{r,\varepsilon} d^r(x',x'')$.
	\end{lemma}
	
	\begin{proof}
		The second claim follows follows from the first and the triangle inequality, so we concentrate on the first claim.
		Also note that, if $0 \le r \le 1$, then the function $t\mapsto t^r$ is subadditive, so the result holds trivially with $c_{r,\varepsilon} =1$.
		Hence, we only need to consider the case of $1\le r < \infty$.
		To do this, suppose first that $a > 0$. 
		Setting $u = \frac{b}{a}$, we need to show that $(1 + u)^r \le (1+\varepsilon) + c_{r,\varepsilon}u^r$, $u \ge 0$, for a suitable constant $c_{r,\varepsilon}$.
		For $0 \le u \le (1+\varepsilon)^{\frac{1}{r}} - 1$ we have $(1 + u)^r \le (1+\varepsilon)$, and for $u \ge (1+\varepsilon)^{\frac{1}{r}} - 1$ we have
		\[
		\frac{(1+u)^r}{u^r} = 
		\left(\frac{1}{u} + 1\right)^r
		\le \left(\frac{1}{(1+\varepsilon)^{\frac{1}{r}} - 1} + 1\right)^r,
		\]
		so the desired inequality holds for $a > 0$ with $c_{r,\varepsilon}$ given by the right-hand side.
		Since this quantity is greater than $1$, it is clear that the desired inequality continues to hold for $a=0$ with this choice of $c_{r,\varepsilon}$.
	\end{proof}
	
	\begin{remark}
		The inequality in Lemma~\ref{lem:Peter_Paul} is akin to
		\eqref{eqn:distance_comparison}	and as such it is sometimes referred to as a ``rob Peter to pay Paul'' inequality or a ``Peter-Paul'' inequality, since for $r\ge0$ and small $\varepsilon$ it shows that we can reduce the coefficient of $d^r(x'',y)$ in the right-hand side of the inequality by increasing the coefficient of $d^r(x',x'')$.
	\end{remark}
	
	\begin{lemma}\label{lem:gp-cts}
		For any $r\ge 0$, the function $g_r: \mathcal{P}_r(X)\times X\to [0,\infty)$ is continuous.
	\end{lemma}
	
	\begin{proof}
		Suppose that $\{(\mu_n,x_n)\}_{n\in\N}$ in $\mathcal{P}_r(X)\times X$ has $(\mu_n,x_n)\to (\mu,x)\in \mathcal{P}_r(X)\times X$.
		Let $\varepsilon > 0$ be arbitrary, and get $c_{r,\varepsilon} > 0$ as in Lemma~\ref{lem:Peter_Paul}.
		Then note that
		\begin{align*}
			\limsup_{n \to \infty}g_{r}(\mu_n,x_n) &= \limsup_{n \to \infty}\int_{X}d^r(x_n,y)\, d\mu_n(y) \\
			&\le \limsup_{n \to \infty}\int_{X}\left(c_{r,\varepsilon}d^r(x_n,x)+(1+\varepsilon)d^r(x,y)\right)\, d\mu_n(y) \\
			&= \limsup_{n \to \infty}\left(c_{r,\varepsilon}d^r(x_n,x)+(1+\varepsilon)\int_{X}d^r(x,y)\, d\mu_n(y)\right) \\
			&= (1+\varepsilon)\limsup_{n \to \infty}\int_{X}d^r(x,y)\, d\mu_n(y) \\
			&= (1+\varepsilon)\int_{X}d^r(x,y)\, d\mu(y) \\
			&= (1+\varepsilon)g_r(\mu,x),
		\end{align*}
		where the penultimate equality holds by the definition of $\tau_{w}^{r}$.
		On the other hand, we also have
		\begin{align*}
			g_r(\mu,x) &= \lim_{n \to \infty}\int_{X}d^r(x,y)\, d\mu_n(y) \\
			&\le \liminf_{n \to \infty}\int_{X}\left(c_{r,\varepsilon}d^r(x,x_n)+(1+\varepsilon)d^r(x_n,y)\right)\, d\mu_n(y) \\
			&= \liminf_{n \to \infty}\left(c_{r,\varepsilon}d^r(x,x_n)+(1+\varepsilon)\int_{X}d^r(x_n,y)\, d\mu_n(y)\right) \\
			&= (1+\varepsilon)\liminf_{n \to \infty}\int_{X}d^r(x_n,y)\, d\mu_n(y) \\
			&= (1+\varepsilon)\liminf_{n \to \infty}g_r(\mu_n,x_n).
		\end{align*}
		Taking $\varepsilon\downarrow 0$ then gives
		\begin{equation}
			\limsup_{n \to \infty}g_r(\mu_n,x_n) \le g_r(\mu,x)\le \liminf_{n \to \infty}g_r(\mu_n,x_n),
		\end{equation}
		which proves the result.
	\end{proof}
	
	Next we introduce another elementary but crucial inequality.
	In the remainder of this section, let $1\le p < \infty$ be fixed.
	
	\begin{lemma}\label{lem:pth-power-difference-bound}
		For any $a,b\ge 0$, we have the inequality $|a^p-b^p|\le p|a-b|(a^{p-1}+b^{p-1})$.
		Consequently, for $x',x'',y\in X$ there is the inequality $|d^p(x',y)-d^p(x'',y)|\le pd(x',x'')(d^{p-1}(x',y)+d^{p-1}(x'',y))$.
	\end{lemma}
	
	\begin{proof}
		The second statement follows from the first and the triangle inequality, so we focus on the first.
		Without loss of generality we may take $a \le b$, hence
		\begin{equation*}
			b^p-a^p = \left|\int_{a}^{b}pu^{p-1}du\right| \le p(b-a)\max\{a^{p-1},b^{p-1}\}\le p(b-a)(a^{p-1}+b^{p-1}),
		\end{equation*}
		as claimed.
	\end{proof}
	
	The preceding inequality allows us to apply a sort of ``renormalization'' trick.
	That is, for $x,z\in X$, and $\mu\in\mathcal{P}_{p-1}(X)$, we have
	\begin{align*}
		&\int_{X}|d^p(x,y)-d^p(z,y)|\, d\mu(y) \\
		&\qquad\le pd(x,z)\left(\int_{X}d^{p-1}(x,y)\, d\mu(y) + \int_{X}d^{p-1}(z,y)\, d\mu(y)\right) < \infty,
	\end{align*}
	which means that the function $d^p(x,\cdot)-d^p(z,\cdot)$ is $\mu$-integrable; consequently, the value
	\begin{equation}
		f_p(\mu,x,z) := \int_{X}(d^p(x,y)-d^p(z,y))\, d\mu(y)
	\end{equation}
	is well-defined.
	As before, let us move certain arguments into the subscript if we wish to emphasize that they are held fixed while others are allowed to vary.
	
	Note that $f_p$ satisfies some elementary but important identities, like  $f_{p,\mu}(x,z) = -f_{p,\mu}(z,x)$ and $f_{p,\mu}(x,z) = f_{p,\mu}(x,x')- f_{p,\mu}(x',z)$.
	However, the validity of $f_p(\mu,x,z) = g_{p}(\mu,x)-g_p(\mu,z)$, which is equivalent to
	\begin{equation}\label{eqn:renorm-separation}
		\int_{X}(d^p(x,y)-d^p(z,y))\, d\mu(y) = \int_{X}d^p(x,y)\, d\mu(y) - \int_{X}d^p(z,y)\, d\mu(y),
	\end{equation}
	is more subtle:
	If $\mu\in\mathcal{P}_p(X)$, then \eqref{eqn:renorm-separation} is indeed true, but, if $\mu\in \mathcal{P}_{p-1}(X)\setminus\mathcal{P}_{p}(X)$, the right side of \eqref{eqn:renorm-separation} is the ill-defined expression $\infty-\infty$.
	Thus, the separation $f_p(\mu,x,z) = g_{p}(\mu,x)-g_p(\mu,z)$ holds only when we have the additional moment $\mu\in \mathcal{P}_{p}(X)$.
	
	\begin{lemma}\label{lem:fp-cts}
		The function $f_p:\mathcal{P}_{p-1}(X)\times X^2\to \R$ is continuous.
	\end{lemma}
	
	\begin{proof}
		Suppose that $\{(\mu_n,x_n,z_n)\}_{n\in\N}$ in $\mathcal{P}_{p-1}(X)\times X^2$ has $(\mu_n,x_n,z_n)\to (\mu,x,z) \in \mathcal{P}_{p-1}(X)\times X^2$.
		We bound
		\begin{align*}
			&|f_p(\mu_n,x_n,z_n) - f_p(\mu,x,z)| \\
			&\qquad \le |f_p(\mu_n,x_n,x)| + |f_p(\mu_n,z_n,z)| + |f_p(\mu_n,x,z)-f_p(\mu,x,z)|, 
		\end{align*}
		so it suffices to show that all terms on the right go to zero.
		
		For the first term, note that Lemma~\ref{lem:pth-power-difference-bound} gives the inequality $|f_p(\mu_n,x_n,x)| \le pd(x_n,x)(g_{p-1}(\mu_n,x_n) + g_{p-1}(\mu_n,x))$.
		Then by Lemma~\ref{lem:gp-cts}, we have $g_{p-1}(\mu_n,x_n) + g_{p-1}(\mu_n,x)\to 2g_{p-1}(\mu,x) < \infty$ while $d(x_n,x)\to 0$, hence $|f_p(\mu_n,x_n,x)|\to 0$.
		The same argument shows $|f_p(\mu_n,z_n,z)|\to 0$.
		
		It only remains to consider the third term.
		For each $n\in\N$ let the random variable $Y_n$ have law $\mu_n$, and let $Y$ have law $\mu$.
		Then set $W_n = d_p(x,Y_n) - d^p(z,Y_n)$ and $W = d^p(x,Y) - d^p(z,Y)$.
		By the continuous mapping theorem \cite[Theorem~4.27]{Kallenberg}, $W_n$ converges to $W$ in distribution.
		Moreover, we have $W_n \le pd(x,z)(d^{p-1}(x,Y_n) + d^{p-1}(z,Y_n)$ by Lemma~\ref{lem:pth-power-difference-bound}.
		Since $pd(x,z)\E[d^{p-1}(x,Y_n) + d^{p-1}(z,Y_n)]\to pd(x,z)\E[d^{p-1}(x,Y) + d^{p-1}(z,Y)]$ by assumption, this implies that $\{W_n\}_{n\in\N}$ is uniformly integrable by \cite[Lemma~4.11]{Kallenberg}.
		In particular, we have $\E[W_n]\to \E[W]$ hence $f_p(\mu_n,x,z)\to f_p(\mu,x,z)$.
		Combining this finishes the proof.
	\end{proof}
	
	\begin{lemma}\label{lem:Fp-characterizations}
		Suppose $\mu\in \mathcal{P}_{p-1}(X),C\subseteq X$ is closed, and $\eta \ge 0$.
		Then, for $x\in C$, the following are equivalent:
		\begin{enumerate}
			\item[(i)] For all $z\in C$ we have $f_p(\mu,x,z)\le \eta$.
			\item[(ii)] For any $o\in X$, we have $f_p(\mu,x,o)\le\min_{z\in C}f_{p}(\mu,z,o) + \eta$.
		\end{enumerate}
		If $\mu\in\mathcal{P}_p(X)$, then the above are also equivalent to
		\begin{enumerate}
			\item[(iii)] $g_p(\mu,x)\le\min_{z\in C}g_{p}(\mu,z) + \eta$.
		\end{enumerate}
	\end{lemma}
	
	\begin{proof}
		Let $x\in C$ be arbitrary.
		Suppose that (i) holds, and let $o\in X$ and $z\in C$ be arbitrary.
		Since $f_p(\mu,x,z) = f_p(\mu,x,o) - f_p(\mu,z,o)$, we can simply rearrange to get (ii).
		Conversely, suppose that (ii) holds, and let $z\in C$ be arbitrary.
		Then take $o=z$ to see that (i) follows.
		Next suppose that we furthermore have $\mu\in\mathcal{P}_p(X)$.
		Suppose (i) holds, and let $z\in C$ be arbitrary.
		Since $f_p(\mu,x,z) = g_p(\mu,x) - g_p(\mu,z)$, we can simply rearrange to conclude (iii).
		Conversely, suppose that (iii) holds.
		Then for any $z\in C$ we again use $f_p(\mu,x,z) = g_p(\mu,x) - g_p(\mu,z)$ and rearrange to see that (i) holds.
	\end{proof}
	
	Now we can define restricted Fr\'echet means in their full generality.
	
	\begin{definition}
		For $\mu\in \mathcal{P}_{p-1}(X),C\subseteq X$ closed, and $\eta \ge 0$, denote the \textit{$\eta$-relaxed $C$-restricted Fr\'echet $p$-mean of $\mu$} as $F_p(\mu,C,\eta)$, which is defined to be the set of all $x\in C$ satisfying any of the equivalent conditions of Lemma~\ref{lem:Fp-characterizations}.
		This set can be empty if no such points exist.
	\end{definition}
	
	Observe now that $\mu\in\mathcal{P}_p(X)$ implies $F_p(\mu,X,0) = F_p(\mu)$ and $F_p(\mu,\supp(\mu),0) = F_p^{\ast}(\mu)$, so the above is indeed a generalization of the definitions given in Section~\ref{sec:intro}.
	Moreover, the argument $\eta \ge 0$ represents an approximation parameter and allows us to study the sets of points which ``nearly'' constitute restricted Fr\'echet means.
	Let us also write $F_p(\mu,\eta) := F_p(\mu,X,\eta)$ and $F_p^{\ast}(\mu,\eta) := F_p(\mu,\supp(\mu),\eta)$ for the sake of simplicity.
	Finally, note also that the solution set of optimization problem in (ii) does not depend on the parameter $o\in X$; this represents a choice of ``origin'' but the arbitrary choice does not affect the minimizers.
	When no choice of $\eta$ or $C$ is specified, we simply refer to these objects as \textit{(relaxed) restricted Fr\'echet mean sets}.
	
	\begin{lemma}\label{lem:Fp-closed-bdd}
		For $\mu\in \mathcal{P}_{p-1}(X),C\subseteq X$ closed, and $\eta \ge 0$, the relaxed restricted Fr\'echet mean set $F_p(\mu,C,\eta)$ is closed in $X$.
		If additionally $\mu\in \mathcal{P}_{p}(X)$, then $F_p(\mu,C,\eta)$ is bounded.
	\end{lemma}
	
	\begin{proof}
		For the first claim, simply write
		\begin{equation*}
			F_p(\mu,C,\eta) = C\cap \bigcap_{z\in C}\{x\in X: f_p(\mu,x,z)\le \eta \},
		\end{equation*}
		and note that each set in the intersection is closed in $X$ as a result of Lemma~\ref{lem:fp-cts}.
		For the second claim, write $I = \min_{x\in C}g_{p,\mu}(x)$.
		Then note that for $x,x'\in F_p(\mu,C,\eta)$, we integrate over \eqref{eqn:distance_comparison_convex} to get
		\begin{align*}
			d^p(x,x') &\le 2^{p-1}\left(\int_{X}d^p(x,y)\, d\mu(y)+\int_{X}d^p(x',y)\, d\mu(y)\right) = 2^{p}(I + \eta) <\infty.
		\end{align*}
		Since the right side is finite and depends only on $\mu$, the claim follows.
	\end{proof}
	
	\begin{remark}\label{rem:Fp-bdd}
		We conjecture, but are unable to prove, that $F_p(\mu,C,\eta)$ is in fact bounded under the weaker moment assumption $\mu\in \mathcal{P}_{p-1}(X)$.
		As we will see later, this is essentially the only obstruction to showing the SLLN holds under the relaxed moment assumption in the stronger senses that we later describe.
	\end{remark}
	
	When $(X,d)$ is compact, Lemma~\ref{lem:gp-cts} and Lemma~\ref{lem:Fp-closed-bdd} imply that $F_p(\mu,C,\eta)$ is non-empty and compact. The next result shows that the same is true even when $(X,d)$ is only assumed to satisfy a weaker condition.
	Recall that a metric space has the {\em Heine-Borel property} if closed, bounded subsets are compact.
	
	\begin{lemma}\label{lem:HB-sep}
		A metric space with the Heine-Borel property is separable.
	\end{lemma}
	
	\begin{proof}
		Let $(X,d)$ be a metric space with the Heine-Borel property, and fix an arbitrary $o\in X$.
		For $n\in\N$, the set $X_n:=\{x\in X: d(x,o) \le n\}$ is a compact metric space.
		Recall that every compact metric space is separable since, for example, one can take as a countable dense set a collection of centers of balls of radii $2^{-k}$ for all $k\in\N$.
		Thus, we have a countable dense subset $D_n\subseteq X_n$ for each $n\in\N$.
		We also have $X= \bigcup_{n\in\N}X_n$, so $D := \bigcup_{n\in\N}D_n$ is a countable dense subset of $X$, hence $X$ is separable.
	\end{proof}
	
	\begin{lemma}\label{lem:Fp-nonempty}
		If $(X,d)$ has the Heine-Borel property, then for all $\mu\in \mathcal{P}_p(X), C\subseteq X$ closed, and $\eta \ge 0$, the relaxed restricted Fr\'echet mean set $F_p(\mu,C,\eta)$ is non-empty and compact.
	\end{lemma}
	
	\begin{proof}
		By Lemma~\ref{lem:Fp-closed-bdd}, it suffices to show that $F_p(\mu,C,\eta)$ is non-empty.
		Further, it suffices to show that $F_p(\mu,C,0)$ is non-empty, since $F_p(\mu,C,0)\subseteq F_p(C,\mu,\eta)$ for all $\eta \ge 0$.
		Adopt the notation introduced in the proof of  Lemma~\ref{lem:Fp-closed-bdd}.
		Consider a sequence $\{x_n\}_{n\in\N}$ in $C$ with $g_{p,\mu}(x_n)\to I$. Then, by integrating \eqref{eqn:distance_comparison_convex},
		\[
		d^p(x_n,x_m) \le 2^{p-1}\left(\int_{X}d^p(x_n,y)\, d\mu(y) + 2^{p-1}\int_{X}d^p(x_m,y)\, d\mu(y)\right) \to 2^pI
		\]
		as $m,n\to\infty$, hence $\{x_n\}_{n\in\N}$ is bounded. Since $(X,d)$ has the Heine-Borel property, there exists a subsequence $\{n_k\}_{k\in\N}$ and a point $x\in C$ with $x_{n_k}\to x$, so, by Fatou's lemma,
		\[
		\int_{X}d^p(x,y)\, d\mu(y) \le \liminf_{k \to \infty}\int_{X}d^p(x_{n_k},y)\, d\mu(y) = I.
		\]
		This shows $x\in F_p(\mu,C,0)$, hence $F_p(\mu,C,\eta)\neq\emptyset$.
	\end{proof}
	
	Next we describe some important notions of convergence of subsets of a given space.
	The study of convergence of subsets of a topological space is deep, and of particular interest to general topologists.
	The study of convergence of subsets of a metric space, where we will focus our attention, is slighly simpler, and has attracted much interest from those in optimization, economics, and stochastic geometry.
	Nonetheless, we still need some of the general theory, which we will recall from \cite[Chapter~5]{Beer}.
	
	Let $\closed(X)$ denote the collection of all closed subset of $X$, and let $\cpt(X)\subseteq \closed(X)$ denote the collection of all non-empty compact subsets of $X$.
	By Lemmas~\ref{lem:Fp-closed-bdd} and \ref{lem:Fp-nonempty}, we see that, in general, $F_p$ takes values in $\closed(X)$, and, when $(X,d)$ has the Heine-Borel property, $F_p$ takes values in $\cpt(X)$, assuming finite moments of sufficiently high order.
	
	Our primary focus is on a few related notions of ``Kuratowski convergence'' in $\closed(X)$; it turns out that these notions are a natural setting for proving strong laws of large numbers for Fr\'echet means.
	To describe these, let $\{C_n\}_{n\in\N}$ be any sequence in $\closed(X)$, and define the closed sets
	
	\begin{align*}
		\underset{n\to\infty}{\kurouter}C_n &= \{x\in X: \liminf_{n\to\infty}d(x,C_n) = 0 \} \\ 
		&= \left\{x\in X:{\begin{matrix}{\mbox{for all open neighborhoods }}U{\mbox{ of }}x,\\ U\cap C_{n}\neq \emptyset \mbox{ for  infinitely many } n\in\N\end{matrix}}\right\} \\
		&= \bigcap_{n=1}^{\infty}\overline{\bigcup_{m=n}^{\infty}C_m} \\
		\underset{n\to\infty}{\kurinner}C_n &= \{x\in X: \limsup_{n \to \infty}d(x,C_n) = 0 \} \\
		&= \left\{x\in X:{\begin{matrix}{\mbox{for all open neighborhoods }}U{\mbox{ of }}x,\\U\cap C_{n}\neq \emptyset {\mbox{ for large enough }}n \in \N\end{matrix}}\right\},
	\end{align*}
	called the \textit{Kuratowski upper limit} and \textit{Kuratowski lower limit}, respectively. These notions were introduced by Kuratowski in \cite[Section~29]{MR0217751}, although he credits P. Painlev\'e for the original idea, so they are sometimes referred to as \textit{Kuratowski-Painlev\'e} upper and lower limits, respectively.
	In general  $\kurinner_{n\to\infty}C_n \subseteq \kurouter_{n\to\infty}C_n$, and, when these sets agree, we write $\kurlimit_{n\to\infty}C_n$ for their common value and we refer to this as the \textit{Kuratowski limit}.
	The upper and lower limit are sometimes called the outer and inner limit, respectively, and the notations $\kurouter$ and $\kurinner$ are typically meant to evoke ``limit superior'' and ``limit inferior'', respectively.
	
	We can also consider Kuratowski upper and lower limits of more general nets of closed sets.
	That is, for any net of closed sets $\{C_{\alpha}\}_{\alpha\in A}$, define
	\begin{align*}
		\underset{\alpha\in A}{\kurouter}C_{\alpha} &= \left\{x\in X:{\begin{matrix}{\mbox{for all open neighborhoods }}U{\mbox{ of }}x,\\ \forall \alpha\in A, \exists\beta\in A,\beta\ge\alpha: U\cap C_{\beta}\neq \emptyset \end{matrix}}\right\} \\
		\underset{\alpha\in A}{\kurinner}C_{\alpha} &= \left\{x\in X:{\begin{matrix}{\mbox{for all open neighborhoods }}U{\mbox{ of }}x,\\\exists \alpha\in A,\forall \beta\in A \mbox{ if }\beta\ge \alpha \mbox{ then } U\cap C_{\beta}\neq \emptyset \end{matrix}}\right\}.
	\end{align*}
	As before we have $\kurinner_{\alpha\in A}C_{\alpha} \subseteq \kurouter_{\alpha\in A}C_{\alpha}$, and, when these sets agree, we write $\kurlimit_{\alpha\in A}C_{\alpha}$ for their common value called the Kuratowski limit.
	We say that $\{C_{\alpha}\}_{\alpha\in A}$ \textit{converges to $C$ in the upper or lower Kuratowski sense} depending on whether we have $\kurouter_{\alpha\in A}C_{\alpha}\subseteq C$ or $\kurinner_{\alpha\in A}C_{\alpha}\supseteq C$, respectively.
	Further, we say that $\{C_{\alpha}\}_{\alpha\in A}$ \textit{converges to $C$ in the (full) Kuratowski sense} if we have $\kurlimit_{\alpha\in A}C_{\alpha}= C$, or equivalently, if we have convergence in both the upper and lower sense.

	Another focus is on ``Hausdorff convergence'' on $\cpt(X)$ which will provide a more quantitative notion of convergence of closed sets.
	For $C,C'\in \cpt(X)$ write $\rho(C,C') = \max_{x\in C}\min_{x'\in C'}d(x,x')$.
	Then take any net $\{C_{\alpha}\}_{\alpha\in A}$ and set $C$ in $\cpt(X)$, and let us say that $\{C_{\alpha}\}_{\alpha\in A}$ \textit{converges to $C$ in the upper, lower, or full Hausdorff sense} depending on whether we have $\lim_{\alpha\in A}\rho(C_{\alpha},C) = 0, \lim_{\alpha\in A}\rho(C,C_{\alpha}) = 0$, or both, respectively.
	
	For the relationships between Kuratowski and Hausdorff convergence, we have the following result, whose simple proof we omit for the sake of brevity.
	We believe that these relationships are likely well-known, but we could not find a reference for them; in general appears that our ``upper'' and ``lower'' notions of convergence are not as commonly studied in the literature.
	
	\begin{lemma}\label{lem:kuratowski-hausdorff-relations}
		For $(X,d)$ any metric space, convergence in the upper, lower, and full Hausdorff sense implies convergence in the upper, lower, and full Kuratowski sense, respectively.
		For $(X,d)$ any compact metric space, convergence in the upper, lower, and full Kuratowski sense implies convergence in the upper, lower, and full Hausdorff sense, respectively.
	\end{lemma}

	Let us now make some remarks about topologies related to these notions of convergence.
	It is easy to show that there is a topology $\uphausdorfftopo$ on $\cpt(X)$, called the \textit{upper Hausdorff topology}, such that a net $\{C_{\alpha}\}_{\alpha\in A}$ and set $C$ in $\cpt(X)$ have $\lim_{\alpha\in A}C_{\alpha} = C$ in $\uphausdorfftopo$ if and only if $\{C_{\alpha}\}_{\alpha\in A}$ converges to $C$ in the upper Hausdorff sense; likewise, there exists a topology $\downhausdorfftopo$ on $\cpt(X)$ called the \textit{lower Hausdorff topology} such that a net $\{C_{\alpha}\}_{\alpha\in A}$ and set $C$ in $\cpt(X)$ have $\lim_{\alpha\in A}C_{\alpha} = C$ in $\downhausdorfftopo$ if and only if $\{C_{\alpha}\}_{\alpha\in A}$ converges to $C$ in the lower Hausdorff sense.
	Then, the \textit{(full) Hausdorff topology} $\hausdorfftopo$ is defined as the join $\hausdorfftopo = \uphausdorfftopo \vee \downhausdorfftopo$, and it satisfies the property that a net $\{C_{\alpha}\}_{\alpha\in A}$ and set $C$ in $\cpt(X)$ have $\lim_{\alpha\in A}C_{\alpha} = C$ in $\hausdorfftopo$ if and only if $\{C_{\alpha}\}_{\alpha\in A}$ converges to $C$ in the (full) Hausdorff sense.
	Both $\uphausdorfftopo$ and $\downhausdorfftopo$ are ``pseudoquasimetrizable''', and $\hausdorfftopo$ is metrized by the familiar $d_H(C,C') = \max\{\rho(C,C'), \rho(C',C)\}$, called the \textit{Hausdorff metric}.

	Contrarily, convergence in the upper, lower, and full Kuratowski senses need not coincide with convergence in some topologies on $\closed(X)$.
	It is known \cite[Theorem~5.2.6]{Beer} that $(X,d)$ being locally compact is necessary and sufficient for there to be a topology $\tau$ on $\closed(X)$ such that for a net $\{C_{\alpha}\}_{\alpha\in A}$ and set $C$ in $\closed(X)$, we have $\lim_{\alpha\in A}C_{\alpha}= C$ in $\tau$ if and only if $\{C_{\alpha}\}_{\alpha\in A}$ converges to $C$ in the (full) Kuratowski sense.
	In this setting, the topology equivalent to Kuratowski convergence is exactly the \textit{Fell topology} $\Fell$ commonly used by stochastic geometers; we say in this case that Kuratowski convergence is \textit{topologizable}.
	It is not known to us when upper and lower Kuratowski convergence are topologizable.
	Since in our case we are interested in the general setting where $(X,d)$ is not assumed to be locally compact, we will not usually discuss convergence in the upper, lower, and full Kuratowski senses as a topological notion.
	
	We also take a moment to explain the naming conventions for these sense of convergence.
	First, note that each ``upper'' sense of convergence has the property that, if a net $\{C_{\alpha}\}_{\alpha\in A}$ converges to $C$ in this sense, then it also converges to $C'$ in the same sense, for any $C'\supseteq C$.
	Likewise, each ``lower'' sense of convergence has the property that, if a net $\{C_{\alpha}\}_{\alpha\in A}$ converges to $C$ this sense, then it also converges to $C'$ in the same sense for any $C'\subseteq C$.
	From this perspective, we see that both ``upper'' and ``lower'' senses of convergence are ``one-sided'', while ``full'' senses of convergence are ``two-sided''.
	When such notions coincide with convergence in topologies, this means that ``full'' topologies are $T_2$ while neither ``upper'' nor ``lower'' topologies are so.
	(When we say that a topology is $T_2$, we mean that any distinct points are contained in disjoint open neighborhoods; most authors refer to such spaces as \textit{Hausdorff} or \textit{separated}, but we prefer to avoid these terms, since they may cause confusion with other terms in the paper.)

	In the next few results, we develop some ``continuity-like'' properties of various maps that will be used later.
	However, as we just remarked, these may not be bona fide continuity statements.
	
	\begin{lemma}\label{lem:support-converge}
		Suppose $(X,d)$ is separable, and that $\{\mu_{n}\}_{n\in\N}$ is a sequence in $\mathcal{P}(X)$ with $\mu_n\to \mu\in \mathcal{P}(X)$ in $\weak$.
		Then we have $\supp(\mu)\subseteq \kurinner_{n\in\N}\supp(\mu_n)$.
	\end{lemma}
	
	\begin{proof}
		First let us prove that, for any Borel measure $\nu$ on $(X,d)$ and any open set $U\subseteq X$, we hav $\nu(U) = 0$ if and only if $U\cap \supp(\nu) = \emptyset$. The ``only if'' direction does not require separability and follows from the definition of the support. For the ``if'' direction, recall that a separable metric space is second-countable, so there exists some countable collection of open sets $\{U_i\}_{i \in \N}$  which form a base for the topology of $X$. Now suppose that an open set $U$ has $U\cap \supp(\nu) = \emptyset$, that is, that any point $x\in U$ has a neighborhood $V_x$ with $\nu(V_x) = 0$. Write each $V_x$ as a countable union $V_x = \bigcup_{i=1}^{\infty}U_{a^x_i}$ where $a^x:\N\to\N$ is some indexing function. Now $U\subseteq \bigcup_{x\in U}V_x = \bigcup_{x\in U} \bigcup_{i=1}^{\infty}U_{a^x_i}$, but this union is actually countable, so we can reindex it as $U \subseteq \bigcup_{i=1}^{\infty}\tilde{U}_i$, where $\{\tilde U_i\}_{i \in \N}$ is a subcollection of $\{U_i\}_{i \in \N}$. Note that $\tilde{U}_i$ is a subset of $V_x$ for some $x$ depending on $i$, so, by construction,  $\nu(\tilde{U}_i) = 0$ for all $i$. 
		Countable subadditivity gives $\nu(U) \le \sum_{i=1}^{\infty}\nu(\tilde{U}_i) = 0$, as claimed.
		
		Now for the main result, suppose that $\{\mu_n\}_{n\in \N}$ is a sequence in $\mathcal{P}(X)$ with $\mu_n\to \mu$ weakly for some $\mu\in \mathcal{P}(X)$.
		Suppose that $U\subseteq X$ is an open set with $U\cap \supp(\mu) \neq\emptyset$. By the equivalence above, this implies $\mu(U) > 0$.
		Now by weak convergence and the Portmanteau theorem \cite[Theorem~4.25]{Kallenberg}, $\liminf_{n \to \infty}\mu_n(U) \ge \mu(U) > 0$.
		This implies that there is some $N>0$ such that $n\ge N$ have $\mu_n(U) > \frac{1}{2}\mu(U) > 0$.
		But, again by the equivalence above, this implies $U\cap \supp(\mu_n)\neq\emptyset$ for $n\ge N$.
		This proves that $\supp(\mu_n)$ eventually intersects $U$, hence $\supp(\mu)\subseteq \kurinner_{n\in\N}\supp(\mu_n)$.
	\end{proof}

	\begin{proposition}\label{prop:joint-Frechet-cts}
		Suppose $\{(\mu_{\alpha},C_{\alpha},\eta_{\alpha})\}_{\alpha\in A}$ is a net in $\mathcal{P}_{p-1}(X)\times \closed(X)\times [0,\infty)$ with $\mu_{\alpha}\to \mu$ in $\tau_{w}^{p-1}$, and $\eta_{\alpha}\to \eta$ in $\R$, and $\kurlimit_{\alpha\in A}C_{\alpha} = C$ for some $(\mu,C,\eta) \in \mathcal{P}_{p-1}(X)\times \closed(X)\times [0,\infty)$.
		Then we have $\kurouter_{\alpha\in A}F_p(\mu_{\alpha},C_{\alpha},\eta_{\alpha}) \subseteq F_p(\mu,C,\eta)$.
	\end{proposition}
	
	\begin{proof}
		For any $x\in \kurouter_{\alpha\in A}F_p(\mu_{\alpha},C_{\alpha},\eta_{\alpha})$, there is a subnet $\{(\mu_{\beta},C_{\beta},\eta_{\beta})\}_{\beta\in B}$ of $\{(\mu_{\alpha},C_{\alpha},\eta_{\alpha})\}_{\alpha\in A}$ and some $x_{\beta}\in F_p(\mu_{\beta},C_{\beta},\eta_{\beta})$ for each $\beta\in B$ such that $x_{\beta}\to x$.
		Since $x_{\beta}\in C_{\beta}$ for $\beta\in B$, we get $x\in \kurouter_{\alpha\in A}C_{\alpha}= C$.
		
		Next we need to show that for any $z\in C$ we have $f_p(\mu,x,z)\le \eta$.
		To do this, note that $z\in C= \kurinner_{\alpha\in A}C_{\alpha}$ implies that there is some $z_{\beta}\in C_{\beta}$ for each $\beta\in B$ such that $z_{\beta}\to z$.
		Then we have by Lemma~\ref{lem:fp-cts}:
		\begin{align*}
			f_p(\mu,x,z) &= \lim_{\beta\in B}f_p(\mu_{\beta},x,z) \\
			&= \lim_{\beta\in B}(f_p(\mu_{\beta},x,x_{\beta}) + f_p(\mu_{\beta},x_{\beta},z_{\beta}) + f_p(\mu_{\beta},z_{\beta},z)) \\
			&\le \lim_{\beta\in B}(f_p(\mu_{\beta},x,x_{\beta}) + \eta_{\beta} + f_p(\mu_{\beta},z_{\beta},z)) = \eta
		\end{align*}
		as claimed.
	\end{proof}
	
	In the remainder of this section, we restrict attention to $\mathcal{P}_{p}(X)$ rather than the larger $\mathcal{P}_{p-1}(X)$.
	However, as we suggested (cf. Remark~\ref{rem:Fp-bdd}), we conjecture that this stronger moment assumption is not necessary, and that the same results all hold in $\mathcal{P}_{p}(X)$.
	
	\begin{lemma}\label{lem:joint-Frechet-cpt}
		Let $(X,d)$ have the Heine-Borel property.
		Suppose $\{(\mu_{\alpha},C_{\alpha},\eta_{\alpha})\}_{\alpha\in A}$ is a net in $\mathcal{P}_{p}(X)\times \closed(X)\times [0,\infty)$ with $\mu_{\alpha}\to \mu$ in $\tau_{w}^{p}$, and $\eta_{\alpha}\to \eta$ in $\R$, and $\kurlimit_{\alpha\in A}C_{\alpha} = C$ for some $(\mu,C,\eta) \in \mathcal{P}_{p-1}(X)\times \closed(X)\times [0,\infty)$.
		Then for any net $\{x_\alpha\}_{\alpha\in A}$ with $x_\alpha\in F_p(\mu_\alpha,C_\alpha,\eta_{\alpha})$, there exists a subnet $\{(\mu_\beta,C_\beta,\eta_{\beta})\}_{\beta \in B}$ of $\{(\mu_\alpha,C_\alpha,\eta_{\alpha})\}_{\alpha \in A}$ and a point $x\in F_p(\mu,C,\eta)$ such that $x_{\beta}\to x$.
	\end{lemma}
	
	\begin{proof}
		By Lemma~\ref{lem:Fp-nonempty}, there exists some $x'\in F_p(\mu,C,\eta)$.
		Since $C\subseteq \kurinner_{\alpha\in A}C_{\alpha}$, there also exists a net $\{z_\alpha'\}_{\alpha\in A}$ with $z_\alpha'\in C_\alpha$ and $z_\alpha'\to x'$.
		Then use \eqref{eqn:distance_comparison_convex} twice to get
		
		\begin{align*}
			d^p(x_{\alpha},x') &\le 2^{p-1}\int_{X}d^p(x_{\alpha},y)\, d\mu_{\alpha}(y) + 2^{p-1}\int_{X}d^p(x',y)\, d\mu_{\alpha}(y) \\
			&\le 2^{p-1}\int_{X}d^p(z_{\alpha}',y)\, d\mu_{\alpha}(y) + 2^{p-1}\int_{X}d^p(x',y)\, d\mu_{\alpha}(y) \\
			&\le 2^{p-1}\int_{X}2^{p-1}(d^p(z_{\alpha}',x')+d^p(x',y))\, d\mu_{\alpha}(y) + 2^{p-1}\int_{X}d^p(x',y)\, d\mu_{\alpha}(y) \\
			&= 2^{2p-2}d^p(z_{\alpha}',x')+2^{p-1}(2^{p-1}+1)\int_{X}d^p(x',y)\, d\mu_{\alpha}(y).
		\end{align*}
		Taking limits gives
		\[
		\limsup_{\alpha\in A}d^p(x_\alpha,x') \le 2^{p-1}(2^{p-1}+1)\int_{X}d^p(x',y)\, d\mu(y).
		\]
		This shows that the points $\{x_{\alpha}\}_{\alpha\in A}$ eventually lie in a closed ball around $x'$.
		By the Heine-Borel property, there exists a subnet $\{x_{\beta}\}_{\beta\in B}$ of $\{x_{\alpha}\}_{\alpha\in A}$ and a point $x\in X$ such that $x_{\beta}\to x$.
		To finish the proof, we just need to show that $x\in F_p(\mu,C,\eta)$.
		To do this, let $z\in C$ be arbitrary, and get a net $\{z_\beta\}_{\beta\in B}$ with $z_{\beta}\in C_{\beta}$ for all $\beta\in B$ and $z_\beta\to z$.
		Now use Lemma~\ref{lem:fp-cts} to get:
		\begin{equation*}
			f_p(\mu,x,z) = \lim_{\beta \in B}f_p(\mu_{\beta},x_{\beta},z_{\beta}) \le \lim_{\beta \in B}\eta_{\beta} =\eta,
		\end{equation*}
		as needed.
	\end{proof}
	
	\begin{proposition}\label{prop:joint-Frechet-Hausdorff-cts}
		Let $(X,d)$ have the Heine-Borel property.
		Suppose that $\{(\mu_{\alpha},C_{\alpha},\eta_{\alpha})\}_{\alpha\in A}$ is a net in $\mathcal{P}_{p}(X)\times \closed(X)\times [0,\infty)$ with $\mu_{\alpha}\to \mu$ in $\tau_{w}^{p}$, and $\eta_{\alpha}\to \eta$ in $\R$, and $\kurlimit_{\alpha\in A}C_{\alpha} = C$ for some $(\mu,C,\eta) \in \mathcal{P}_{p}(X)\times \closed(X)\times [0,\infty)$.
		Then we have $\lim_{\alpha\in A}\rho(F_p(\mu_{\alpha},C_{\alpha},\eta_{\alpha}),F_p(\mu,C,\eta)) = 0$.
	\end{proposition}
	
	\begin{proof}
		For any subnet $\{(\mu_\beta,C_\beta,\eta_{\beta})\}_{\beta \in B}$ of $\{(\mu_\alpha,C_\alpha,\eta_{\alpha})\}_{\alpha \in A}$, get a net $\{x_\beta\}_{\beta\in B}$ with $x_\beta\in F_p(\mu_{\beta},C_{\beta},\eta_{\beta})$ for all $\beta\in B$.
		By Lemma~\ref{lem:joint-Frechet-cpt}, there exists a further subnet $\{(\mu_\gamma,C_\gamma,\eta_{\gamma})\}_{\gamma \in \Gamma}$ of $\{(\mu_\beta,C_\beta,\eta_{\beta})\}_{\beta \in B}$ and a point $x\in F_p(\mu,C,\eta)$ such that $x_{\gamma}\to x$, and this implies
		\begin{equation*}
			\lim_{\gamma \in \Gamma}\rho(F_p(\mu_{\gamma},C_{\gamma},\eta_{\gamma}),F_p(\mu,C,\eta))= 0.
		\end{equation*}
		Since this holds for any subnet, we have established  $\lim_{\alpha\in A}\rho(F_p(\mu_{\alpha},C_{\alpha},\eta_{\alpha}),F_p(\mu,C,\eta))= 0$, and this is exactly the desired conclusion.
	\end{proof}
	
	\begin{remark}
		For a net $\{C_\alpha\}_{\alpha}$ in $\closed(X)$ with $\kurouter_{\alpha\in A}C_\alpha\subseteq C$, the assumption that every $C_\alpha$ is non-empty and compact is, in general, not sufficient to guarantee that $\lim_{\alpha\in A}\rho(C_{\alpha},C) =0$.
		To see this, consider $X=\R$ with the usual metric, the sequence $C_n = \{n\}$, and $C = \{0\}$. Thus, Proposition~\ref{prop:joint-Frechet-Hausdorff-cts} is truly stronger than Proposition~\ref{prop:joint-Frechet-cts}.
	\end{remark}
	
	\begin{remark} 
		We are considering $1 \le p < \infty$, but there is a natural extension of this theory to the case of $p=\infty$.
		Define $\mathcal{P}_{\infty}(X)$ to be the set of all Borel probability measures on $X$ with bounded support, and define, for $\mu\in \mathcal{P}_{\infty}(X)$ and $x\in X$,
		\[
		g_{\infty}(\mu,x) = ||d(x,\cdot)||_{L^{\infty}(X,\mu)}= \esssup\{d(x,\cdot), \mu\}.
		\]
		Here, for a measurable space $(S,\mathcal{S},\nu)$ and a measurable function $h:S\to \R$, we write $\esssup\{h,\nu\}$ for the infimum over all $c\in\R$ such that $\nu(\{y\in X: h(y) > c \}) = 0$.
		Then, $F_{\infty}$, which is sometimes called the \textit{circumcenter}, can be easily to shown to share most of the properties with $F_p$ for $1 \le p < \infty$; the proofs are similar with only minor modifications.
		Therefore, most of the main theorems of Section~\ref{sec:results} also hold for $p=\infty$.
		However, for the sake of brevity, we will not treat the $p=\infty$ case in the remainder of this paper and return to our assumption that $1 \le p < \infty$.
	\end{remark}
	
	It is likely that the conclusion of Proposition~\ref{prop:joint-Frechet-Hausdorff-cts} holds under the following slightly weaker hypotheses:
	Instead of requiring that the space $(X,d)$ has the Heine-Borel property, one can require that the space
	\begin{equation*}
		\overline{\bigcup_{\substack{\alpha\in A \\ \alpha \ge \beta}}F_p(\mu_{\alpha},C_{\alpha},\eta_{\alpha})}
	\end{equation*}
	has the Heine-Borel property for some $\beta\in A$, in the metric inherited from $(X,d)$, 
	Indeed, such is the hypothesis of similar results in \cite[Theorem~A.4]{Huckemann} and \cite[Theorem~3.5]{Schoetz2}.
	However, it appears difficult to verify this property without the a priori assumption $(X,d)$ itself has the Heine-Borel property; as remarked in \cite{Schoetz2}, this ``may be of similar difficulty as showing convergence of Fr\'{e}chet means directly''.
	Thus, we focus on the case that $(X,d)$ has the Heine-Borel property.
	
	To close this section, we detail the connection between Fr\'echet $p$-means and optimal transport.
	Let $W_p$ denote the $p$-Wasserstein metric on $\mathcal{P}_p(X)$; that is
	\begin{equation*}
		W_p(\mu,\nu) = \inf_{\pi\in \Gamma(\mu,\nu)}\left(\int_{X\times X}d^p(x,y)\, d\pi(x,y)\right)^{1/p}
	\end{equation*}
	where $\Gamma(\mu,\nu)$ denotes the collection of all couplings of $\mu$ and $\nu$ on $X\times X$.
	Now define, for each $x\in X$, the sets $K_{x} = \{\mu\in \mathcal{P}_p(X): x\in F_p(\mu)\} = \{\mu\in \mathcal{P}_p(X): g_{p}(\mu,x) \le g_{p}(\mu,x') \text{ for all } x'\in X \}$.
	It is clear that $g_p(\mu,x) = W_p^p(\mu,\delta_x)$, so it follows by the monotonicity of $a\mapsto a^p$ that the set $F_p(\mu)$ is just the set of points $x\in X$ which minimize their $W_p$-distance to the measure $\mu$.
	In other words, the collection $\mathbf{K}_p = \{K_{x}\}_{x\in X}$ is a \textit{Voronoi diagram} in $(\mathcal{P}_p(X),W_p)$ with landmark points $\{\delta_x\}_{x\in X}$.
	By Lemma~\ref{lem:Fp-nonempty} it follows that, if $(X,d)$ has the Heine-Borel property, then $\mathbf{K}_p$ is covering of $\mathcal{P}_p(X)$.
	In fact, these sets are rather nice:
	
	\begin{lemma}
		Fix $1 \le p < \infty$ and write $\mathbf{K}_p = \{K_x\}_{x\in X}$.
		Then, for each $x\in X$, the set $K_{x}$ is a closed, convex subset of $\mathcal{P}_p(X)$.
	\end{lemma}
	
	\begin{proof}
		To see that $K_{x}$ is closed, suppose that $\{\mu_n\}_{n\in\N}$ are in $K_{x}$ and that $\mu_n\to \mu$ holds in $\weakp$.
		Then for any $x'\in X$ we have $g_{p,x}(\mu) = \lim_{n\to\infty}g_{p,x}(\mu_n) \le \lim_{n\to\infty}g_{p,x'}(\mu_n) = g_{p,x'}(\mu)$ by Lemma~\ref{lem:gp-cts}, hence $\mu\in K_{x}$.
		To see that $K_{x}$ is convex, use the linearity of the function $g_{p,x}$ to get that if $\mu_1,\mu_2$ are in $K_{x}$ and $\alpha\in (0,1)$, then, for any $x'\in X$,
		\begin{align*}
			g_{p,x}(\alpha\mu_1+(1-\alpha)\mu_2) &=  \alpha g_{p,x}(\mu_1) + (1-\alpha) g_{p,x}(\mu_2) \\
			&\le  \alpha g_{p,x'}(\mu_1) + (1-\alpha) g_{p,x'}(\mu_2) \\
			&= g_{p,x'}(\alpha\mu_1+(1-\alpha)\mu_2).
		\end{align*}
		This finishes the proof.
	\end{proof}
	
	It also interesting to define the sets $K_{x}^{\ast} = \{\mu\in\mathcal{P}_p(X) : x\in F_p^{\ast}(\mu)\} = \{\mu\in\mathcal{P}_p(X): x\in\supp(\mu)\text{ and } g_p(\mu,x) \le g_p(\mu,x')\text{ for all } x'\in \supp(\mu)\}$ in order to try to get an analogous geometric picture for the support-restricted Fr\'echet $p$-means.
	Again by Lemma~\ref{lem:Fp-nonempty}, the collection $\mathbf{K}_p^{\ast} = \{K_{x}^{\ast}\}_{x\in X}$ covers $\mathcal{P}_p(X)$ whenever $(X,d)$ has the Heine-Borel property.
	However, there is no longer a nice interpetation of these sets as a Voronoi diagram.
	Even worse is that the sets in $\mathbf{K}_p^{\ast}$ can fail to be closed and convex.
	To see this, consider $X =\{-1,0,1\}$ with the metric inherited from the real line, and let $p>1$ be arbitrary.
	To see that $K_{-1}^{\ast}$ is not convex, define $\mu_1 = \frac{1}{2}\delta_{-1}+\frac{1}{2}\delta_{0}$ and $\mu_2 = \frac{1}{2}\delta_{-1}+\frac{1}{2}\delta_{1}$.
	Then, $\mu_1,\mu_2\in K_{-1}^\ast$ but $\frac{1}{2}\mu_1+\frac{1}{2}\mu_2 = \frac{1}{2}\delta_{-1}+\frac{1}{4}\delta_{0} + \frac{1}{4}\delta_{1}\notin K_{-1}^\ast$.
	To see that $K_{0}^{\ast}$ is not closed, define $\mu_n = \frac{1}{2}(1-\frac{1}{n})\delta_{-1} + \frac{1}{2}(1-\frac{1}{n})\delta_{1} + \frac{1}{n}\delta_{0}$ for $n\in\N$.
	Then note that $\mu_n\in K_{0}^{\ast}$ for all $n\in\N$, but $\mu_n$ converges weakly to $\frac{1}{2}\delta_{-1}+\frac{1}{2}\delta_{1}\notin K_{0}^{\ast}$.
	
	These observations show that the support-restricted Fr\'echet $p$-means are significantly less well-behaved than the unrestricted Fr\'echet $p$-means.
	This phenomenon can be traced primarily back to the observation that the supports of two measures can differ greatly while their $p$-Wasserstein distance remains small.
	It is somewhat surprising, then, that, despite this difficulty, our main results hold for both Fr\'echet $p$-means and support-restricted Fr\'echet $p$-means alike.
	
	A cartoon to visualize $\mathbf{K}_p$ and $\mathbf{K}_p^{\ast}$ is given in Figure~\ref{fig:voronoi}.
	The value of this geometric picture is that it leads to a sort of ``dual'' version of the problem:
	Whereas previously $\mu\in\mathcal{P}_p(X)$ was fixed and we considered the set of $x\in X$ which minimized the function $g_{p,\mu}$, we can instead fix $x\in X$ and consider the set of $\mu\in\mathcal{P}_p(X)$ such that $x$ is a minimizer of $g_{p,\mu}$.
	The results above show that, in the case of Fr\'echet $p$-means on a compact metric space, the determination of $x\in F_p(\mu)$ is equivalent to checking whether a given measure lies in a certain compact convex set.
	However, this cartoon can also be somewhat misleading for general $(X,d)$: The regions of $\mathbf{K}_p$ and $\mathbf{K}_p^{\ast}$ need not be polytopes and they may have empty interior.
	
	\begin{figure}
		\begin{tikzpicture}
			\definecolor{paleRed}{HTML}{ffadaf}
			\definecolor{paleGreen}{HTML}{b1e6b3}
			\definecolor{paleBlue}{HTML}{8cc8ff}
			\definecolor{Green}{HTML}{27912b}
			
			\fill[fill=paleRed] (0.8,2.07846) -- (0,3.4641) -- (-1,1.7320) -- cycle;
			\fill[fill=paleGreen] (1.2,1.38564) -- (2,0) -- (0,0) -- cycle;
			\fill[fill=paleBlue] (-2,0) -- (0,0) -- (1.2,1.38564) -- (0.8,2.07846) -- (-1,1.73205) -- cycle;
			
			\draw[ultra thick,color=blue] (-1,1.73205) -- (-2,0) -- (0,0);	
			\draw[ultra thick,color=red] (-1,1.7320) -- (0,3.4641) -- (0.8,2.07846);
			\draw[ultra thick,color=blue] (0.8,2.07846) -- (1.2,1.38564);
			\draw[ultra thick,color=Green] (1.2,1.38564) -- (2,0) -- (0,0);
			
			\node at (-1.8,2.3) {$\mathcal{P}_p(X)$};
			
			\node at (-2.3,0) {$\delta_{x_1}$};
			\node at (0.4,3.4641) {$\delta_{x_2}$};
			\node at (2.4,0) {$\delta_{x_3}$};
			
			\node[red] at (0,2.4) {$K_{{x_2}}$};
			\node[Green] at (1.1,0.4) {$K_{{x_3}}$};
			\node[blue] at (-0.3,0.9) {$K_{{x_1}}$};
		\end{tikzpicture} \begin{tikzpicture}
			\definecolor{paleRed}{HTML}{ffadaf}
			\definecolor{paleGreen}{HTML}{b1e6b3}
			\definecolor{paleBlue}{HTML}{8cc8ff}
			\definecolor{Green}{HTML}{27912b}
			
			\fill[fill=paleRed] (0.8,2.07846) -- (0,3.4641) -- (-1,1.7320) -- cycle;
			\fill[fill=paleGreen] (1.2,1.38564) -- (2,0) -- (0,0) -- cycle;
			\fill[fill=paleBlue] (-2,0) -- (0,0) -- (1.2,1.38564) -- (0.8,2.07846) -- (-1,1.73205) -- cycle;
			
			\draw[ultra thick,color=blue] (-1,1.73205) -- (-2,0) -- (0,0);	
			\draw[ultra thick,color=red] (-1,1.7320) -- (0,3.4641) -- (1,1.73205);
			
			\draw[ultra thick,color=Green] (1,1.73205) -- (2,0) -- (0,0);
			
			\node at (-1.8,2.3) {$\mathcal{P}_p(X)$};
			
			\node at (-2.3,0) {$\delta_{x_1}$};
			\node at (0.4,3.4641) {$\delta_{x_2}$};
			\node at (2.4,0) {$\delta_{x_3}$};
			
			\node[red] at (0,2.4) {$K_{{x_2}}^{\ast}$};
			\node[Green] at (1.1,0.4) {$K_{{x_3}}^{\ast}$};
			\node[blue] at (-0.3,0.9) {$K_{{x_1}}^{\ast}$};
		\end{tikzpicture}
		\caption{The space of measures $\mathcal{P}_p(X)$ and its covering by $\mathbf{K}_p$ (left) and $\mathbf{K}_p^{\ast}$ (right), where $X=\{x_1,x_2,x_3\}$.}
		\label{fig:voronoi}
	\end{figure}
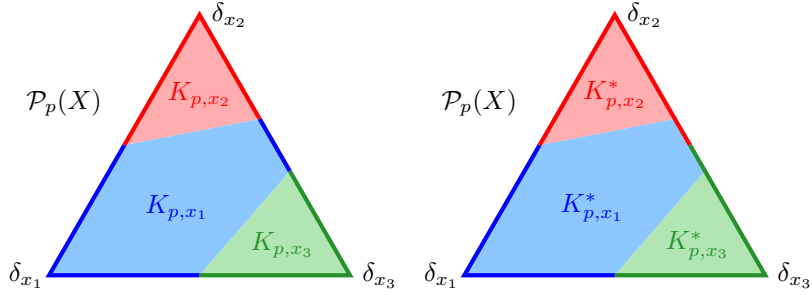
	
	\section{Measurability Concerns}\label{sec:measurability}
	
	For some metric space $(X,d)$ and some $\mu\in \mathcal{P}_{p-1}(X)$, let $(\Omega,\F,\P)$ denote a probability space on which independent, identically distributed (i.i.d.) random elements $Y_1,Y_2,\ldots $ with common distribution $\mu$ are defined. 
	As above, $\bar \mu_n$, $n \in \N$, is the empirical distribution of $Y_1, \ldots, Y_n$. 
	
	An interesting technical point is that it is not immediately clear whether sets like $\{\kurouter_{n \to \infty}F_p(\bar \mu_n) \subseteq F_p(\mu)\}$, or other sets describing other senses of convergence of related objects, are in $\F$.
	In all of our main results, these will turn out to be supersets of events in $\F$ that have probability one, and hence they are at least measurable with respect to the $\P$-completion of $\F$. In this section we address the natural question of whether they are actually measurable with respect to the possibly incomplete $\sigma$-algebra $\F$. More generally, we clarify the measurability of certain objects of interest with respect to some natural measurability structures on $\closed(X)$.
	
	It is standard to endow $\closed(X)$ with the $\sigma$-algebra generated by the sets $\{\{C \in \closed(X) : C \cap U \ne \emptyset\}: U\subseteq X \mbox{ open}\}$, which is called the \textit{Effros $\sigma$-algebra} and is denoted $\effros(X)$ \cite[Section~3.3]{Srivastava}.
	By a \textit{random set} we mean a map $\Omega\to \closed(X)$ that is $\F/\effros(X)$-measurable.
	For a topology $\tau$ on any set, let us also write $\mathcal{B}(\tau)$ for the Borel $\sigma$-algebra on this set generated by $\tau$.
	
	In this section only, we find it most useful to study the unrestricted Fr\'echet $p$-means and the support-restricted Fr\'echet $p$-means separately.
	We also specialize to $\eta = 0$, although this is only for the sake of simplicity and the generalization to $\eta > 0$ is straightforward.
	
	\begin{lemma}\label{lem:F-mble}
		Suppose $1 \le p < \infty$ and that $(X,d)$ is separable and locally compact.
		Then, the map $F_p:\mathcal{P}_p(X)\to \closed(X)$ is $\mathcal{B}(\weakp)/\effros(X)$-measurable.
	\end{lemma}
	
	\begin{proof}
		We prove the statement in two steps.
		First, let $C(X;\R)$ denote the space of all continuous, real-valued functions on $X$, and let $\pointwise$ denote the topology of pointwise convergence on $C(X;\R)$.
		Let $o\in X$ be arbitrary, and consider the map $f:\mathcal{P}_p(X)\to C(X;\R)$ defined by $f(\mu) := f_{p}(\mu,\cdot , o)$; by Lemma~\ref{lem:fp-cts} we know that $f(\mu):X\to\R$ is indeed a continuous function. It follows from the definitions that $f:(\mathcal{P}_p(X),\weakp)\to (C(X;\R),\pointwise)$ is continuous, hence measurable.
		
		Next we show that the map $M:C(X;\R)\to \closed(X)$ defined by 
		$M(f) := \arg\min_{x\in X}f(x)$ is $\mathcal{B}(\pointwise)/\effros(X)$-measurable. Let $D = \{x_1,x_2,\ldots\} \subseteq X$ be a countable dense set, and use local compactness to get, for each $n \in \N$, some radius $r_n > 0$ such that the closure of the ball $B_{r_n}(x_n)$ is compact. Now the
		collection $\{B_{r}(x_n): r \in (0,r_n)\cap \Q, n \in \N\}$ is a basis for the metric topology of $(X,d)$, so the sets $\{\{C \in \closed(X) : C  \cap  B_r(x_n) \ne \emptyset\},  r \in (0,r_n)\cap \Q , n \in \N\}$ generate the $\sigma$-algebra $\effros(X)$. 
		It thus suffices to check that $M^{-1}(\{C \in \closed(X) : C  \cap  B_r(x_n) \ne \emptyset\})\in \mathcal{B}(\pointwise)$ holds for all $r \in (0,r_n)\cap \Q$, $n \in \N$. Observe that for each $n \in \N$ and $r \in (0,r_n)\cap \Q$
		\begin{align*}
			&M^{-1}(\{C \in \closed(X) : C  \cap  B_r(x_n) \ne \emptyset\}) \\
			&\qquad = \left\{f \in C(X) : \arg\min_{x\in X} f(x) \cap  B_r(x_n) \ne \emptyset\right\} \\
			&\qquad=
			\bigcup_{s\in (0,r)\cap\Q}\left\{f \in C(X) : \inf_{x \in \overline{B}_s(x_n)}f(x) = \inf_{x \in X}f(x)  \right\} \\
			&\qquad=
			\bigcup_{s\in (0,r)\cap\Q}\left\{f \in C(X) : \inf_{x \in \overline{B}_s(x_n)\cap D}f(x) = \inf_{x \in D}f(x)  \right\},
		\end{align*}
		where we used the fact that a continuous function $f$ 
		attains its minimum on the compact ball $\overline{B}_s(x_n)$. 
		Now note for $x_0 \in X$ that the point evaluation $\delta_{x_0}:(C(X),\mathcal{B}(\pointwise))\to (\R,\mathcal{B}(\R))$ 
		defined via $\delta_{x_0}(f) := f(x_0)$ is continuous and hence measurable, so
		\begin{align*}
			&\left\{f \in C(X) : \inf_{x \in \overline{B}_s(x_n)\cap D}f(x) = \inf_{x \in D}f(x)  \right\} \\
			&\qquad= \bigcap_{x \in D}\bigcup_{x' \in \overline{B}_s(x_n)\cap D}\left\{f \in C(X) : f(x') \le f(x) \right\}
		\end{align*}
		is in $\mathcal{B}(\pointwise)$. This proves the claim.
	\end{proof}
	
	\begin{lemma}\label{lem:F-restr-mble}
		For $1\le p < \infty$, the map $F_p^{\ast}(\bar \mu_n):\Omega\to \closed(X)$ is $\F/\effros(X)$-measurable.
	\end{lemma}
	
	\begin{proof}
		For any open set $U\subseteq X$, we have
		\begin{align*}
			&\{F_p^{\ast}\cap U = \emptyset\} \\
			&\qquad= \{\forall x\in \supp(\bar \mu_n)\cap U, \ \exists  x'\in \supp(\bar \mu_n) \text{ with } f_{p}(\bar \mu_n,x,x') > 0 \} \\
			&\qquad= \bigcap_{j=1}^{n}\left(\{Y_j\notin  U\} \cup \bigcup_{k=1}^{n} \{f_{p}(\bar \mu_n,Y_j,Y_k) > 0\}\right) \\
			&\qquad= \bigcap_{i=j}^{n}\left(\{Y_j\notin  U\} \cup \bigcup_{k=1}^{n} \left\{\frac{1}{n}\sum_{i=1}^{n}(d^p(Y_j,Y_i)  - d^p(Y_k,Y_i)) > 0\right\}\right).
		\end{align*}
		Since $\{Y_i\in U\}$ and $\{\frac{1}{n}\sum_{i=1}^{n}(d^p(Y_j,Y_i)  - d^p(Y_k,Y_i)) > 0\}$ are both in $\F$, and  since the union and intersection appearing above are finite, it  follows that $\{F_p^{\ast}(\bar \mu_n)\cap U = \emptyset\}\in\F$, so the result is proved.
	\end{proof}
	
	It is easily shown that $\bar \mu_n:\Omega\to \mathcal{P}_p(X)$ is $\F/\mathcal{B}(\weakp)$-measurable, so Lemma~\ref{lem:F-mble} implies that $F_p(\bar \mu_n):\Omega\to \closed(X)$ is $\F/\effros(X)$-measurable whenever $(X,d)$ is separable and locally compact. Compare this to Lemma~\ref{lem:F-restr-mble} in which we showed directly that  $F_p^{\ast}(\bar \mu_n):\Omega\to \closed(X)$ is $\F/\effros(X)$-measurable under no assumptions on the metric space $(X,d)$. It is not known to us what conditions are sufficient to ensure that the map $F_p^{\ast}:\mathcal{P}_{p}(X)\to \closed(X)$ is in fact $\mathcal{B}(\weakp)/\effros(X)$-measurable.
	
	\begin{lemma}\label{lem:kurouter-mble}
		Suppose that $(X,d)$ is separable and locally compact.  Assume that $C_n:\Omega\to \closed(X)$, $n \in \N$, are $\F/\effros(X)$-measurable. Then, $\kurouter_{n\to\infty}C_n:\Omega\to \closed(X)$ is $\F/\effros(X)$-measurable.
	\end{lemma}
	
	\begin{proof}
		Let $D = \{x_n\}_{n\in\N}$ be a countable dense subset of $X$, and, using local compactness, for each $k \in \N$, get some $r_k > 0$ such that the closure of the ball $B_{r_k}(x_k)$ is compact. As above, showing $\left\{\kurouter_{n \to \infty} C_n \cap B_r(x_k) \ne \emptyset\right\}\in \F$ for all $k \in \N$ and all $r \in \Q \cap (0,r_k]$ gives the claim. Note first that
		\[
		\left\{\underset{n\to\infty}{\kurouter} C_n \cap B_r(x_k) \ne \emptyset\right\}
		=
		\bigcup_{s \in (0,r)\cap \Q} 
		\left\{\bigcap_{n=1}^\infty \overline{\bigcup_{m=n}^\infty C_m} \cap \overline{B}_s(x_k) \ne \emptyset\right\}.
		\]
		Now note that, for fixed $s\in (0,r_k)$, the sequence 
		\[
		\left\{\overline{\bigcup_{m=n}^\infty C_m} \cap \overline{B}_{s}(x_k)\right\}_{n =1}^{\infty}
		\]
		is a nonincreasing sequence of compact sets.  Therefore, by Cantor's intersection theorem \cite[Theorem~5.1]{Kelley_55}, 
		\[
		\left\{\bigcap_{n=1}^\infty \overline{\bigcup_{m=n}^\infty C_m} \cap \overline{B}_s(x_k) \ne \emptyset\right\}
		=
		\bigcap_{n=1}^\infty \left\{\overline{\bigcup_{m=n}^\infty C_m} \cap \overline{B}_s(x_k) \ne \emptyset\right\}.
		\]
		Again using compactness, for $0 < s < r_k$,
		\[
		\left\{\overline{\bigcup_{m=n}^\infty C_m} \cap \overline{B}_s(x_k) \ne \emptyset\right\}
		=
		\bigcap_{t \in (s,r_k)\cap \Q} \left\{\bigcup_{m=n}^\infty C_m \cap  B_t(x_k) \ne \emptyset\right\}.
		\]
		Lastly, for $t\in (0,r_k)$
		\[
		\left\{\bigcup_{m=n}^\infty C_m \cap  B_t(x_k) \ne \emptyset\right\}
		=
		\bigcup_{m=n}^\infty \left\{C_m \cap  B_t(x_k) \ne \emptyset\right\}.
		\]
		Since $\{C_m \cap  B_t(x_k) \ne \emptyset\}\in \effros(X)$ by definition, the result is proved.
	\end{proof}
	
	Interestingly, there does not appear to be a direct adaptation to the proof of Lemma~\ref{lem:kurouter-mble} which shows that $\kurinner_{n\to\infty}C_n$ is $\F/\effros(X)$-measurable whenever $\{C_n\}_{n\in\N}$ are so. This is because we relied on the characterization of the Kuratowski upper limit via an explicit formula, and, as shown in \cite{MR0089398}, there is no formula for the Kuratowski lower limit which is composed of countably many of the operations of union, intersection, and closure.
	
	\begin{lemma}\label{lem:hausdorff-mble}
		If $C:\Omega\to\cpt(X)$ is $\F/\effros(X)$-measurable and $C'\in\cpt(X)$ is fixed, then $\rho(C,C'):\Omega\to\R$ is $\F/\mathcal{B}(\R)$-measurable.
	\end{lemma}
	
	\begin{proof}
		Write $f:X\to\R$ for $f(x) = \min_{x'\in C'}d(x,x')$ and observe that $f$ is continuous. Thus, for any $\alpha\in\R$,
		\[
		\left\{\rho(C,C') \le \alpha\right\} = \left\{\max_{x\in C}f(x)\le \alpha\right\} = \left\{C\cap f^{-1}((\alpha,\infty)) = \emptyset\right\},
		\]
		and the right side is in $\F$ by the definition of $\effros(X)$.
	\end{proof}
	
	Combining Lemmas~\ref{lem:F-mble}, \ref{lem:F-restr-mble}, \ref{lem:kurouter-mble}, and \ref{lem:hausdorff-mble} shows that, when $1\le p < \infty$ and $(X,d)$ is separable and locally compact, the sets $\{\kurouter_{n \to \infty}F_p(\bar \mu_n)\subseteq F_p(\mu)\}$, $\{\kurouter_{n \to \infty}F_p^{\ast}(\bar \mu_n)\subseteq F_p^{\ast}(\mu)\}$, and $\{\lim_{n\to\infty}\rho(F_p(\bar \mu_n), F_p(\mu)) = 0\}$ are events in the $\sigma$-algebra $\F$, whether or not $\F$ is complete.
	Similarly, Lemmas~\ref{lem:F-restr-mble} and \ref{lem:hausdorff-mble} show that $\{\lim_{n \to \infty}\rho(F_p^\ast(\bar \mu_n),F_p^\ast(\mu)) = 0\}$ is in $\F$ for any metric space $(X,d)$.
	In general, it seems that measurability properties of the lower (and full) senses of convergence are harder to establish; it is not known to us when $\{C_n\}_{n\in\N}$ being $\F/\effros(X)$-measurable implies that $\kurinner_{n\to\infty}C_n$ is $\F/\effros(X)$-measurable, nor when $\rho(C,C_n)$ is $\F/\mathcal{B}(\R)$-measurable
	
	\section{Proofs of the Main Probabilistic Results}\label{sec:results}
	
	In this section, which is divided into several subsections, we prove the main probabilistic results of the paper.
	Some subsections conclude with concrete examples used to illustrate various aspects of the theory.
	
	\subsection{One-Sided SLLNs}
	
	First, we prove the strong laws of large numbers (SLLNs) for empirical restricted Fr\'echet means under various one-sided senses of convergence.
	For the setting, we let $(\Omega,\F,\P)$ denote a complete probability space on which independent, identically-distributed (i.i.d.) random elements $Y_1,Y_2,\ldots$ with common distribution $\mu\in\mathcal{P}(X)$ are defined.
	As before, let $\bar \mu_n$ for $n \in \N$, denote the empirical measure of the first $n$ samples.
	
	We begin with two auxiliary results.
	
	\begin{proposition}
		\label{prop:pw-LLN}
		Suppose that $(X,d)$ is a separable and that $r\ge 0$ is such that $\mu\in \mathcal{P}_r(X)$. 
		Then, $\bar \mu_n\to\mu$ holds in $\tau_{w}^{r}$ almost surely.
	\end{proposition}
	
	\begin{proof}
		Since $(X,d)$ is separable, the main result in \cite{Varadarajan} gives that $\bar \mu_n\to\mu$ in $\weak$ almost surely.
		From Lemma~\ref{lem:one_x_all_x} it therefore suffices to show that $g_{r,x}(\bar \mu_n)$ converges almost surely to $g_{r,x}(\mu)$ for some fixed $x \in X$.
		But $g_{r,x}(\bar \mu_n)= \frac{1}{n} \sum_{i=1}^n d^r(x, Y_i)$, so this is immediate from the classical strong law of large numbers.
	\end{proof}
	
	\begin{proposition}\label{prop:iid-supports-converge}
		Suppose $(X,d)$ is separable.
		Then $\kurlimit_{n \to \infty}\supp(\bar \mu_n) = \supp(\mu)$.
	\end{proposition}
	
	\begin{proof}
		Again, \cite{Varadarajan} guarantees $\bar \mu_n\to\mu$ in $\weak$ almost surely, so by Lemma~\ref{lem:support-converge} we have $\supp(\mu)\subseteq \kurinner_{n\in\N}\supp(\bar \mu_n)$ almost surely.
		To see $\kurouter_{n\in\N}\supp(\bar \mu_n)\subseteq \supp(\mu)$ almost surely, note that $x\notin \supp(\mu)$ implies that there is an open set $U\subseteq X$ with $x\in U$ such that $\mu(U) = 0$.
		Then, $\bar \mu_n(U) = 0$ for all $n\in\N$ almost surely, hence $x\notin \kurouter_{n\in\N}\supp(\bar \mu_n)$ almost surely, which completes the proof.
	\end{proof}

	Now we turn to the main result.
	We let $1\le p < \infty$ be fixed, as before.
	In this first part, we show how the convergence results of Section~\ref{sec:topology} can be used to easily prove SLLNs for various kinds of restricted Fr\'echet means.
	
	\begin{theorem}\label{thm:Frechet-SLLN}
		Suppose that $(X,d)$ is separable and that $\mu\in \mathcal{P}_{p-1}(X)$.
		Then for any random elements $\{(C_n,\eta_n)\}_{n\in\N}$ and $(C,\eta)$ in $\closed(X)\times [0,\infty)$, we have
		\begin{equation}
			\kurouter_{n \to \infty}F_p(\bar \mu_n,C_n,\eta_n) \subseteq F_p(\mu,C,\eta)
		\end{equation}
		almost surely on $\{\kurlimit_{n \to \infty}C_n = C \text{ and }\eta_n\to \eta\}$.
	\end{theorem}
	
	\begin{proof}
		Immediate from Proposition~\ref{prop:pw-LLN} and Proposition~\ref{prop:joint-Frechet-cts}.
	\end{proof}

	Next we turn to the strong laws of large numbers in the upper Hausdorff upper sense, which provide a more quantitative notion of convergence.
	As we remarked before, we require the moment assumption $\mu\in\mathcal{P}_p(X)$ because of the corresponding assumption in Proposition~\ref{prop:joint-Frechet-Hausdorff-cts}, although we believe that this assumption is not necessary.
	As evidence of this, we point out that the special case where $C_n = C = X$ for $n\in\N$ and where $\{\eta_n\}_{n\in\N}$ is non-random with $\eta_n\to 0$ as $n\to\infty$ is proved in \cite[Theorem~2]{Schoetz2}.
	
	\begin{theorem}\label{thm:Frechet-SLLN-hausdorff}
		Suppose that $(X,d)$ has the Heine-Borel property, and that $\mu\in \mathcal{P}_{p}(X)$.
		Then for any random elements $\{(C_n,\eta_n)\}_{n\in\N}$ and $(C,\eta)$ in $\closed(X)\times [0,\infty)$, we have
		\begin{equation}
			\lim_{n\to\infty}\rho(F_p(\bar \mu_n,C_n,\eta_n),F_p(\mu,C,\eta))= 0
		\end{equation}
		almost surely on $\{\kurlimit_{n \to \infty}C_n = C \text{ and }\eta_n\to \eta\}$.
	\end{theorem}
	
	\begin{proof}
		Immediate from Proposition~\ref{prop:pw-LLN}, Proposition~\ref{prop:joint-Frechet-Hausdorff-cts}, and Lemma~\ref{lem:HB-sep}.
	\end{proof}
	
	\begin{corollary}
		Suppose that $(X,d)$ is separable and that $\mu\in \mathcal{P}_{p-1}(X)$.
		Then for any real-valued random variables $\{\eta_n\}_{n\in\N}$ and $\eta$, we have
		\begin{align*}
			\kurouter_{n \to \infty}F_p(\bar \mu_n,\eta_n) &\subseteq F_p(\mu,\eta) \text{ and } \\
			\kurouter_{n \to \infty}F_p^{\ast}(\bar \mu_n,\eta_n) &\subseteq F_p^{\ast}(\mu,\eta)
		\end{align*}
		almost surely on $\{\eta_n\to \eta\}$.
		If furthermore $(X,d)$ has the Heine-Borel property and $\mu\in\mathcal{P}_{p}(X)$, then
		\begin{align*}
			\rho(F_p(\bar \mu_n,\eta_n),F_p(\mu,\eta)) &\to 0\text{ and } \\
			\rho(F_p^{\ast}(\bar \mu_n,\eta_n),F_p^{\ast}(\mu,\eta)) &\to 0
		\end{align*}
		almost surely on $\{\eta_n\to \eta\}$.
	\end{corollary}
	
	\begin{proof}
		The two claims of first part follow from Theorem~\ref{thm:Frechet-SLLN}, respectively by taking $C_n = C = X$ and by taking $C_n = \supp(\bar \mu_n)$ and $C = \supp(\mu)$ and applying Proposition~\ref{prop:iid-supports-converge}.
		The two claims of the second part follow from Theorem~\ref{thm:Frechet-SLLN-hausdorff} and the same cases.
	\end{proof}
	
	By taking $\eta_n =0$ for all $n\in\N$, the result above recovers the fact that $\mu\in \mathcal{P}_{p-1}(X)$ is sufficient for $\kurouter_{n \to \infty}F_p(\bar \mu_n) \subseteq F_p(\mu)$ and $\kurouter_{n \to \infty}F_p^{\ast}(\bar \mu_n) \subseteq F_p^{\ast}(\mu)$ almost surely when $(X,d)$ is separable, and that $\mu\in \mathcal{P}_{p}(X)$ is sufficient for $\rho(F_p(\bar \mu_n),F_p(\mu))\to 0$ and $\rho(F_p^{\ast}(\bar \mu_n),F_p^{\ast}(\mu))\to 0$ almost surely when $(X,d)$ has the Heine-Borel property.
	However, we emphasize that the result is much more general than this, since we make no assumption about the random elements $\{\eta_n\}_{n\in\N}$ or their dependence structure.
	In particular, it seems interesting for statistical applications to choose these parameters adaptively as a function of the data; our result guarantees the correct long-term behavior as long as these adaptive parameters converge to some limit almost surely.
	A similar opportunity is available in choosing the domains $\{C_n\}_{n\in\N}$ adaptively as a function of the data, although it seems more challenging to choose these in a way that they almost surely have a limit in the Kuratowski sense.

	Next we note that, in the setting of Theorem~\ref{thm:Frechet-SLLN}, we have shown that, almost surely,
	\begin{equation}\label{eqn:inclusions}
		\kurinner_{n\to\infty}F_p(\bar \mu_n) \overset{(i)}{\subseteq} \kurouter_{n\to\infty}F_p(\bar \mu_n) \overset{(ii)}{\subseteq} F_p(\mu).
	\end{equation}
	This raises the question of whether the inclusions $(i)$ and $(ii)$ can be made strict and if they can be made with equality; the question of whether (ii) can be made strict was previously stated as open question (3) of \cite{HuckemannOberwolfach}, and we provide a positive answer.
	In fact, we give a very strong form of a positive answer by showing that all outcomes are possible.
	We state all of our results for unrestricted Fr\'echet means, but we believe that simple adaptations can make the same results true for support-restricted Fr\'echet means.
	
	\begin{example}[equality in both (i) and (ii)]
		Take $X=\R$ and $p=2$, and let $\mu\in \mathcal{P}_1(\R)$ be arbitrary.
		Then $F_p(\mu) = \{\int_{\R}y\, d\mu(y)\}$, and $F_p(\bar \mu_n) = \{\frac{1}{n}\sum_{i=1}^{n}Y_i\}$ for all $n\in\N$.
		Hence by the classical SLLN we get
		\begin{equation*}
			d_{H}\left(F_p(\bar \mu_n),F_p(\mu)\right) = \left|\int_{\R}y\, d\mu(y)-\frac{1}{n}\sum_{i=1}^{n}Y_i\right| \to 0
		\end{equation*}
		as $n\to\infty$ almost surely.
		By Lemma~\ref{lem:kuratowski-hausdorff-relations}, this implies
		\begin{equation*}
			\kurinner_{n\to\infty}F_p(\bar \mu_n) = \kurouter_{n\to\infty}F_p(\bar \mu_n) = F_p(\mu) = \left\{\int_{\R}y\, d\mu(y)\right\}
		\end{equation*}
		almost surely
	\end{example}
	
	\begin{example}[strictness in (i) and equality in (ii)]\label{ex:discrete-uniform}
		Consider $X = \{1,2,\ldots, m\}$ for $m\ge 2$, and write $d$ for the discrete metric, $d(x,y) = \ind\{x\neq y\}$.
		Let $\mu$ be the uniform measure on $X$.
		Observe that $F_p(\mu) = X$ and also $F_p(\bar \mu_n) = \arg \max_{x\in X}N_x^n$, where $N_x^n := \#\{1 \le i \le n : Y_i = x\}$.
		Moreover, because $X$ is discrete
		\begin{align*}
			\underset{n\to\infty}{\kurinner}F_p(\bar \mu_n) &= \bigcup_{n=1}^\infty \bigcap_{k=n}^\infty F_p(\bar \mu_k) \\
			\underset{n\to\infty}{\kurouter}F_p(\bar \mu_n) &= \bigcap_{n=1}^\infty \bigcup_{k=n}^\infty F_p(\bar \mu_k).
		\end{align*}
		The Kuratowski lower and upper limits of $F_p(\bar \mu_n)$ are both invariant under finite permutations of the 
		samples and so, by the Hewitt-Savage zero-one law, it follows that they are both constant almost surely. 
		Because the probability measure $\mu$ is invariant under permutations of the elements of $X$, 
		each of $\kurinner_{n\to\infty}F_p(\bar \mu_n)$ and $\kurouter_{n\to\infty}F_p(\bar \mu_n)$ must almost surely equal either $\emptyset$ or $X$.
		
		For the Kuratowski upper limit, note that, because $X$ is finite, there must exist $Z \in X$ such that $Z \in \arg \max_{x\in X}N_x^n$ for infinitely many $n \in \N$.
		Thus, the Kuratowski upper limit is nonempty almost surely and hence equal to $X$ almost surely.
		However, in order for the Kuratowski lower limit to be nonempty, and hence equal to $X$ almost surely, we must have $\arg \max_{x\in X}N_x^n = X$ for all $n \in \N$ sufficiently large. 
		This is equivalent to $N_1^n = \cdots = N_m^n = \frac{n}{m}$ for all $n \in \N$ sufficiently large.
		This is impossible, since the equalities can only occur when $n \in \N$ is divisible by $m$.
		Therefore,  $\kurinner_{n\to\infty}F_p(\bar \mu_n) = \emptyset$ almost surely and $\kurouter_{n\to\infty}F_p(\bar \mu_n) = F_p(\mu) = X$ almost surely.
	\end{example}

	
	
	\begin{example}[strictness in both (i) and (ii)]\label{ex:both-strict}
		For $m\ge 4$ put $X = \{0,1,2,\ldots, m\}$. 
		Consider the function $d:X\times X\to \R_{\ge 0}$ such that $d(w,w) = 0$ for all $w \in X$, $d(x,y) = 1$ for all $x,y\in\{1,2,\ldots, m\}$ with $x\neq y$, and $d(0,z) = (1-\frac{1}{m})^{1/p}$ for all $z\in \{1,2,\ldots, m\}$.
		It is easy to see that $d$ is a metric.
		Define the probability measure $\mu$ on $X$ to be the uniform distribution on $\{1,2,\ldots, m\}$. 
		
		Note that
		\[
		\int_{X}d^p(x,y)\, d\mu(y) = \frac{m-1}{m} = \int_{X}d^p(0,y)\, d\mu(y)
		\]
		for any $x\in \{1,2,\ldots, m\}$ and hence $F_p(\mu) = X$.
		
		Set $N_x^n := \#\{1 \le i \le n : Y_i = x\}$ for $x \in \{1,\ldots,m\}$.
		Observe for $x \in \{1, \ldots, m\}$ that
		\[
		\int_{X}d^p(x,y)\, d \bar \mu_n(y) 
		= 
		1 - \frac{N_x^n}{n}
		\]
		and
		\[
		\int_{X}d^p(0,y)\, d \bar \mu_n(y) 
		= 
		1 - \frac{1}{m}.
		\]
		Therefore, for $x \in \{1,\ldots,m\}$, we have $x\in F_p(\bar \mu_n)$ if and only if $x\in \arg\max_{y\in \{1,\ldots, m\}}N_y^n$ and $N_x^n \ge \frac{n}{m}$, but the first condition implies the second and so the first condition is necessary and sufficient for $x\in F_p(\bar \mu_n)$.
		Similarly, $0\in F_p(\bar \mu_n)$ if and only if $\frac{n}{m} \ge N_x^n$ for all $x \in \{1,\ldots,m\}$, and this is equivalent to $N_1^n = \cdots = N_m^n = \frac{n}{m}$.
		As in the previous example, we also have that $\kurinner_{n\to\infty}F_p(\bar \mu_n)$ and $\kurouter_{n\to\infty}F_p(\bar \mu_n)$ are invariant under finite permutations of the samples, so the Hewitt-Savage zero-one law implies that $\kurinner_{n\to\infty}F_p(\bar \mu_n)$ and $\kurouter_{n\to\infty}F_p(\bar \mu_n)$ are almost surely constant. 
		
		Using these observations, we claim that $\kurouter_{n\to\infty}F_p(\bar \mu_n) = \{1,\ldots, m\}$ holds almost surely.
		To see this, note that $(N_1^n,\ldots, N_m^n)$ are just the counts of a multinomial random variable with $m$ outcomes of equal likelihood, so, applying Stirling's formula (only when $n$ is a multiple of $m$), yields
		\[
		\P(0\in F_p(\bar \mu_n)) = \frac{n!}{\left(\left(\frac{n}{m}\right)!\right)^mm^n} \sim \frac{m^{\frac{m}{2}}}{(2\pi)^{\frac{m-1}{2}}}\cdot \frac{1}{n^{\frac{m-1}{2}}}.
		\]
		Since $m \ge 4$, this sequence is summable in $n \in m\N$, hence it follows from the Borel-Cantelli lemma that, with probability one, $0\in F_p(\bar \mu_n)$ holds only finitely often.
		This implies that $0\notin \kurouter_{n\to\infty}F_p(\bar \mu_n)$ almost surely. 
		From the symmetry properties of $\mu$ it follows that either $\kurouter_{n\to\infty}F_p(\bar \mu_n) = \emptyset$ almost surely or $\kurouter_{n\to\infty}F_p(\bar \mu_n) = \{1,\ldots,m\}$ almost surely.
		Since $\{1,\ldots, m\}$ is finite, there must exist some $Z\in \{1,\ldots, m\}$ with $Z\in \arg\max_{x\in \{1,\ldots, m\}}N_x^n$ for infinitely many $n \in \N$, hence $\kurouter_{n\to\infty}F_p(\bar \mu_n) \ne \emptyset$ almost surely.
		Therefore, $\kurouter_{n\to\infty}F_p(\bar \mu_n) = \{1,\ldots,m\}$ almost surely.
		
		We also claim that $\kurinner_{n\to\infty}F_p(\bar \mu_n) = \emptyset$ holds almost surely.
		Since  $\kurinner_{n\to\infty}F_p(\bar \mu_n) \subseteq \kurouter_{n\to\infty}F_p(\bar \mu_n)$, we know from the above that $\kurinner_{n\to\infty}F_p(\bar \mu_n) \subseteq \{1,\ldots,m\}$ almost surely.
		From the symmetry properties of $\mu$ it follows that either $\kurinner_{n\to\infty}F_p(\bar \mu_n) = \emptyset$ almost surely or $\kurinner_{n\to\infty}F_p(\bar \mu_n) = \{1,\ldots,m\}$ almost surely.
		As in the previous example, the latter requires that $N_1^n = \cdots = N_m^n = \frac{n}{m}$ for sufficiently large $n \in \N$, but this is impossible since these equalities can only occur when $m$ divides $n \in \N$. (Of course, we have actually shown above the stronger statement that, almost surely, $N_1^n = \cdots = N_m^n = \frac{n}{m}$ occurs for only finitely many $n \in \N$.)
		
		Summarizing these conclusions, we have
		\begin{align*}
			\kurinner_{n\to\infty}F_p(\bar \mu_n) &= \emptyset, \\
			\kurouter_{n\to\infty}F_p(\bar \mu_n) &= \{1,2,\ldots, m\}, \\
			F_p(\mu) &= \{0,1,2,\ldots, m\},
		\end{align*}
		almost surely.
	\end{example}
	
	Lastly, we need an example where equality is achieved in $(i)$ and strictness is achieved in $(ii)$.
	In fact, it turns out that such an example is impossible if $(X,d)$ is a finite metric space:
	Since $X$ is finite and each $F_p(\bar \mu_n)$ is non-empty, there must exist some $x\in X$ that is in infinitely many $F_p(\bar \mu_n)$.
	Then $x\in \kurouter_{n\to\infty}F_p(\bar \mu_n) \subseteq F_p(\mu)$, and, by Theorem~\ref{thm:full-SLLN-finite}, the equality in $(i)$ implies $F_p(\mu)= [x]_{\mu}$.
	So, for any other point $x'\in F_p(\mu)$, we have $d(x,\cdot) = d(x',\cdot)$ almost surely, and this implies $x'\in \kurouter_{n\to\infty}F_p(\bar \mu_n)$ almost surely, hence $\kurouter_{n\to\infty}F_p(\bar \mu_n) \supseteq F_p(\mu)$ almost surely. Nonetheless, we have a concrete example in the following.
	
	\begin{example}[equality in (i) and strictness in (ii)]\label{ex:irrationals}
		Set $X=(([0,1] \setminus \Q) \cup \{0,1\}$; that is, $X$ consists of the irrational numbers contained in the interval $[0,1]$ along with the endpoints $0$ and $1$.
		Take $d$ to be the restriction to $X$ of the usual metric on $\R$.
		It is clear that this is a separable metric space, since, for example, $(\sqrt{2} \Q)\cap (0,1)$ is a countable dense subset.
		(In fact, $X$ is not too exotic: It is clearly a $G_\delta$ subset of the Polish space $[0,1]$, so it is itself a Polish space \cite[Theorem~2.2.1]{Srivastava}.)
		
		Now fix an irrational number $0 < t < 1$, and let $\mu$ be the probability measure $(1-t) \delta_0 + t \delta_1$.
		We claim that, with probability one, $\bar \mu_n$ is of the form $\frac{n-N}{n} \delta_0 + \frac{N}{n} \delta_1$ for some random $1 \le N \le n-1$ for all $n \in \N$ sufficiently large.
		Indeed, this is equivalent to the statement that there exist positive integers $i$ and $j$ with $X_i = 0$ and $X_j = 1$, and this follows again from the second Borel-Cantelli lemma since $\sum_{n\in\N}\P(X_n = 0) = \sum_{n\in\N}\P(X_n = 1) = \infty$.
		So, for $p>1$, with probability one, $F_p(\bar \mu_n) = \emptyset$ for all $n \in \N$ sufficiently large;
		this is because, in this case, the function $g_{p,\bar \mu_n}:X\to \R$ has near-minimizers in $X$ but no true minimizers in $X$.
		Therefore, $\kurouter_{n\to\infty}F_p(\bar \mu_n) = \emptyset$ almost surely.
		Since $\kurinner_{n\to\infty}F_p(\bar \mu_n) \subseteq \kurouter_{n\to\infty}F_p(\bar \mu_n)$, this implies that $\kurinner_{n\to\infty}F_p(\bar \mu_n) = \emptyset$ almost surely as well.
		Therefore, $\kurlimit_{n\to\infty}F_p(\bar \mu_n) = \emptyset$ almost surely.
		Finally, we have $F_p(\mu) = \{t\}$ for $p > 1$, since $g_{p,\mu}:[0,1]\to\R$ has, by convexity, a unique minimizer in $[0,1]$, and this minimizer lies in $X$.
	\end{example}
	
	\subsection{Further One-Sided Limit Theorems}
	
	In this subsection we exploit the true power of the tools developed in Section~\ref{sec:topology}, which is that they can be applied to more general settings in which one wants to ``descend'' limit theorems in a space of measures to limit theorems for restricted Fr\'echet means.
	In particular, there is nothing constraining us to the case of i.i.d. samples nor to the case of strong laws of large numbers.
	As an illustration of this point, we easily prove two results with a different flavor than the rest of our main results.
	For the sake of simplicity, we focus just on unrestricted Fr\'echet $p$-means.
	
	First observe that for any random sets $\{C_n\}_{n\in\N}$ and $C$ in $\cpt(X)$, it is clear that $\rho(C_n,C)\to 0$ almost surely implies $\rho(C_n,C)\to 0$ \textit{in probability}, in the sense that for all $\varepsilon > 0$ we have $\P\{\rho(C_n,C)\ge\varepsilon\} \to 0$ as $n\to\infty$.
	(Recall from Lemma~\ref{lem:hausdorff-mble} that $\{\rho(C_n,C) \ge \varepsilon\}\in\F$, so these probabilities are well-defined.)
	It is interesting to try to understand the rate of decay of these probabilities.
	A useful framework for studying the decay of tail probabilities like those above is large deviations theory, the basics of which can be found in \cite{LargeDeviations}.
	
	The basic idea of our analysis is relatively standard: We begin with a known large deviations principle for empirical measures $\{\bar \mu_n\}_{n\in\N}$ in the space $(\mathcal{P}_{p}(X),\weakp)$, and we use the ``contraction principle'' for the continuous function $F_p:(\mathcal{P}_{p}(X),\weakp)\to (\cpt(X),\uphausdorfftopo)$ to deduce a large deviations principle for the Fr\'echet means $\{F_p(\bar \mu_n)\}_{n\in\N}$ in the space $(\cpt(X),\uphausdorfftopo)$.
	Fortunately, we have \cite{SanovWasserstein} which provides the necessary large deviations principle for $\{\bar \mu_n\}_{n\in\N}$ in $(\mathcal{P}_{p}(X),\weakp)$.
	Unfortunately, the standard formulation of the contraction principle \cite[Theorem~4.2.1]{LargeDeviations} does not apply, since it relies on both the domain and the codomain being $T_2$; ultimately, this is needed in order for compact sets to be closed.
	Our setting certainly does not have this property since, as we have seen before, the topology $\uphausdorfftopo$ is not $T_2$.
	Nonetheless, we can slightly modify the proof of the standard contraction principle to get an analogous but slightly weaker result.
	
	Towards proving the result, we define the function
	\begin{equation}
		H(\nu \, | \, \mu) = \begin{cases}
			\int_{X}\log(\frac{d\nu}{d\mu}) \, d\nu, &\text{ if } \nu \ll \mu, \\
			\infty, &\text{ else},
		\end{cases}
	\end{equation}
	called the \textit{relative entropy} for probability measures $\mu,\nu\in\mathcal{P}(X)$; the relative entropy appears in many large deviations problems involving empirical measures, including our own.
	Also note that, for this part only, we need to invoke topology explicitly, so we recall that $\uphausdorfftopo$ represents the topology on $\cpt(X)$ such that $\{C_n\}_{n\in\N}$ and $C$ in $\cpt(X)$ have $\lim_{n\to\infty}C_n = C$ iff $\lim_{n\to\infty}\rho(C_n,C)= 0$.
	Then we have the following result, which represents a simple step towards understanding the desired rate of decay.
	
	\begin{theorem}\label{thm:Fp-LDP}
		Suppose that $(X,d)$ is a compact metric space.
		Then, the function $I_{p,\mu}:(\cpt(X),\uphausdorfftopo)\to [0,\infty]$ defined via
		\begin{equation}
			I_{p,\mu}(C) = \inf\{H(\nu \, | \, \mu): \nu\in\mathcal{P}(X), \,   F_p(\nu) \supseteq C \},
		\end{equation}
		has compact sublevel sets and has the property that, for any set $A\in\mathcal{B}(\uphausdorfftopo)$,
		\begin{equation}\label{eqn:Fp-LDP-UB}
			\limsup_{n\to\infty}\frac{1}{n}\log\P\{F_p(\bar \mu_n) \in A\} \le -\inf\{I_{p,\mu}(C) : C\in \bar A \}.
		\end{equation}
		Consequently, for any $\varepsilon\ > 0$, there exist constants $c_1,c_2 > 0$ (depending on $(X,d)$, $\mu$, $p$, and $\varepsilon$) such that $\P\{\rho(F_p(\bar \mu_n),F_p(\mu))\ge\varepsilon\} \le c_1\exp(-c_2n)$ for all $n\in\N$.
	\end{theorem}
	
	\begin{proof}
		First we note that a compact metric space must be complete and separable: completeness is immediate from compactness, and for each $n\in\N$, there is a finite cover by balls of radius $2^{-n}$, and the centers of these balls form a countable dense set.
		Therefore, $(X,d)$ is a complete, separable metric space, so, by \cite[Theorem~1.1]{SanovWasserstein}, the random measures $\{\bar \mu_n\}_{n\in\N}$ satisfy a large deviations principle in $(\mathcal{P}_{p}(X),\weakp)$ with good rate function (that is, lower semicontinuous and with compact sublevel sets) given by $\nu\mapsto H(\nu \, | \, \mu)$.
		Also, Proposition~\ref{prop:joint-Frechet-Hausdorff-cts} guarantees that the map $F_p$ is continuous (hence measurable), so for any $A\in\mathcal{B}(\uphausdorfftopo)$ we have
		\begin{align*}
			&\limsup_{n \to \infty}\frac{1}{n}\log\P\{F_p(\bar \mu_n) \in A\} \\
			&\qquad =\limsup_{n \to \infty}\frac{1}{n}\log\P\{\bar \mu_n \in F_p^{-1}(A)\} \\
			&\qquad\le -\inf\{H(\nu \, | \, \mu): \nu\in \overline{F_p^{-1}(A)} \} \\
			&\qquad= -\inf\{H(\nu \, | \, \mu): \nu\in F_p^{-1}(\bar A) \} \\
			&\qquad= -\inf\{H(\nu \, | \, \mu): \nu\in\mathcal{P}_p(X), F_p(\nu) \in \bar A\} \\
			&\qquad= -\inf\{\inf\{H(\nu \, | \, \mu): \nu\in\mathcal{P}_p(X), F_p(\nu) = C\} : C \in \bar A\} \\
			&\qquad\le -\inf\{I_{p,\mu}(C): C \in \bar A\},
		\end{align*}
		where the first inequality is exactly the large deviations upper bound for the empirical measures.
		This proves the desired bound.
		
		Next we show that $I=I_{p,\mu}$ has compact level sets.
		To do this, write $\Psi_{I}(\alpha) = \{C\in\cpt(X):I(C) \le \alpha \}$ and $\Psi_{H}(\alpha)= \{\nu\in\mathcal{P}(X):H(\nu \, | \, \mu) \le \alpha \}$.
		On the one hand, if $\nu\in \Psi_H(\alpha)$, then we clearly have $I(F_p(\nu)) \le \alpha$, hence $\Psi_I(\alpha) \supseteq F_p(\Psi_{H}(\alpha))$.
		On the other hand, if $C\in \Psi_I(\alpha)$, then there exists a sequence $\{\nu_n\}_{n\in\N}$ in $\mathcal{P}(X)$ with $F_p(\nu_n) \supseteq C$ and $H(\nu_n|\mu)\downarrow I(C)$.
		Since $\Psi_I(\alpha+1)$ is compact, there exists some subsequence $\{n_k\}_{k\in\N}$ and some $\nu\in \mathcal{P}(X)$ such that  $\nu_{n_k}\to \nu$ in $\weakp$.
		Thus,  $H(\nu \, | \, \mu) \le \liminf_{k\to\infty}H(\nu_{n_k}|\mu) = I(C) \le \alpha$ by the lower semicontinuity of $H$ and also $C \subseteq \kurouter_{k \to \infty}F_p(\nu_{n_k}) \subseteq F_p(\nu)$ by the continuity of $F_p$.
		This proves $C\in F_p(\Psi_I(\alpha))$, hence $\Psi_I(\alpha) \subseteq F_p(\Psi_{H}(\alpha))$.
		We have therefore shown $\Psi_I(\alpha) = F_p(\Psi_{H}(\alpha))$, and since the continuous image of a compact set is compact, this shows that $\Psi_I(\alpha)$ is compact.
		
		For the last claim, note that for any $\varepsilon > 0$, we can apply \eqref{eqn:Fp-LDP-UB} to the closed set $A = \{C\in \cpt(X): \rho(C,F_p(\mu)) \ge \varepsilon\}$ to get
		\begin{align*}
			&\limsup_{n \to \infty}\frac{1}{n}\log\P\{\rho(F_p(\bar \mu_n),F_p(\mu)\ge \varepsilon\} \\
			&\qquad\le -\inf\{H(\nu \, | \, \mu): \nu\in\mathcal{P}_p(X), \rho(F_p(\nu),F_p(\mu))\ge \varepsilon \} := -c_0.
		\end{align*}
		Then, it suffices to show that $c_0 > 0$.
		To see this, write $D = \max_{x,x'\in X}d(x,x')$ for the diameter of $(X,d)$, and let $q\in (1,\infty]$ be the exponent conjugate to $p\in [1,\infty)$.
		By the transportation inequality \cite[Theorem 6.13]{Villani} and Pinsker's inequality \cite[Lemma~2.5]{Tsybakov}, we have, for any $\nu\in\mathcal{P}(X)$,
		
		\begin{align*}
			W_p(\nu,\mu) &\le 2^{\frac{1}{q}}\left(\int_{X}d^p(x,y)\,d|\nu-\mu|(y)\right)^{\frac{1}{p}} \\
			&\le 2^{\frac{1}{q}}D\left(\int_{X}d|\nu-\mu|(y)\right)^{\frac{1}{p}} \\
			&= 2^{\frac{1}{q}}D\|\nu-\mu\|_{\text{TV}}^{\frac{1}{p}} \\
			&\le 2^{\frac{1}{q}+\frac{1}{2}}D(H(\nu \, | \, \mu))^{\frac{1}{p}}.
		\end{align*}
		Now assume that we had $c_0 = 0$.
		This would imply that there exists a sequence $\{\nu_n\}_{n\in\N}$ in $F_p^{-1}(A)$ with $H(\nu_n|\mu)\to 0$, hence $W_p(\nu_n,\mu)\to 0$ by the above.
		But $W_p$ metrizes $\weakp = \weak$ whenever $(X,d)$ is compact and separable \cite[Theorem~6.8]{Villani}, hence $\nu_n\to \mu$ in $\weak$.
		Since $F_p^{-1}(A)$ is closed in $(\mathcal{P}(X),\weak)$, this implies $\mu \in F_p^{-1}(A)$ which is clearly a contradiction. Therefore, $c_0 > 0$.
		Finally, choosing $c_2 = c_0/2$ and $c_1$ sufficiently large gives the claim.
	\end{proof}
	
	While in Theorem~\ref{thm:Fp-LDP} we showed exponential asymptotic decay of the tail probabilities, it seems interesting (and important for applications) to try to understand finite-sample bounds.
	This task is undertaken in \cite{AGP,Schoetz1} for the case where $F_p(\mu)$ is assumed to be a singleton; in \cite{Schoetz1} there is also some indication of the way that the methods therein may be adapted to the general setting.
	These works seem to suggest that the finite-sample decay of these tail probabilities depends heavily on the geometry of $(X,d)$, thus it is interesting that Theorem~\ref{thm:Fp-LDP} holds without any such assumptions. 
	
	Next we consider an example of a convergence theorem with a different dependence structure than i.i.d. samples.
	Of course, Markov chains provide a simple extension.
	
	\begin{theorem}\label{thm:Fp-ergodic-MC}
		Suppose that $(X,d)$ is a finite metric space.
		Let $P$ be an irreducible and aperiodic Markov transition kernel from $X$ to itself, and let $\nu_{\infty}$ be its stationary distribution.
		For any initial distribution $\nu$, let $(Y_i)_{i\in \N}$ denote the Markov chain started in $\nu$ and following $P$, and for $n\in\N$ let $\bar \nu_n$ denote the empirical measure of $Y_1,\dots, Y_n$.
		Then, $\rho(F_p(\bar \nu_n),F_p(\nu_{\infty}))\to 0$ almost surely as $n\to\infty$.
	\end{theorem}
	
	\begin{proof}
		It is classical that in this setting  $\bar \nu_n\to \nu_{\infty}$ in $\weak$ almost surely, and hence that, for any $x\in X$, 
		\[
		\frac{1}{n}\sum_{i=1}^{n}d^p(x,Y_i) \to \int_{X}d^p(x,y)\, d\nu_{\infty}(y)
		\]
		almost surely.
		This is equivalent to $g_{p,x}(\bar \nu_n)\to g_{p,x}(\nu_{\infty})$ almost surely, so $\bar \nu_n\to \nu_{\infty}$ in $\weakp$ almost surely.
		Therefore, the result follows from Lemma~\ref{prop:joint-Frechet-Hausdorff-cts}.
	\end{proof}
	
	\subsection{Two-Sided SLLNs}
	
	In this subsection we return to the setting of SLLNs for i.i.d. samples, focusing now on the question of when we have a SLLN under some ``two-sided'' notion of convergence.
	It turns out that all of these results will rely on $(X,d)$ having the Heine-Borel property, and hence all convergences are bona fide convergences in some topology on $\cpt(X)$.
	By Theorem~\ref{thm:Frechet-SLLN-hausdorff}, the almost sure convergence will always be at least as strong as $\uphausdorfftopo$.
	In particular, we will provide various sufficient conditions for  almost sure convergence in the topology $\Fell|_{\cpt(X)}\vee \uphausdorfftopo$, the join of $\uphausdorfftopo$ with the restriction of Fell topology $\Fell$ to $\cpt(X)$; for simplicity let us abuse notation and write $\Fell\vee\uphausdorfftopo$ in place of $\Fell|_{\cpt(X)}\vee\uphausdorfftopo$.
	Of course, this convergence is at least as strong as convergence in $\Fell$ alone.
	
	We have seen that proving SLLNs (and other limit theorems) in ``upper topologies'' amounts to combining limit theorems for empirical measures with the continuity-like results of Section~\ref{sec:topology}.
	In contrast, proving SLLNs for ``lower topologies'' requires more probabilistic and geometric reasoning.
	For the sake of simplicity, we focus on the case of unrestricted and support-restricted Fr\'echet means with $\eta = 0$.
	We begin by introducing a notion of equivalence in metric measure spaces; we do not know if this notion has already been studied in metric geometry or elsewhere.
	
	\begin{definition}\label{def:equivalence}
		Let $(X,d)$ be any metric space.  Fix $\mu\in \mathcal{P}(X)$. For $x,x'\in X$, write $x\leadsto_{\mu} x'$ if, for sufficiently small $\varepsilon>0$, there exists a measurable map $\phi:B_{\varepsilon}(x)\to B_{\varepsilon}(x')$ such that for any $z\in B_{\varepsilon}(x)$ we have that $d(\phi(z),\cdot) = d(z,\cdot)$ holds $\mu$-almost everywhere.
		Write $x\sim_{\mu}x'$ if $x\leadsto_{\mu} x'$ and $x'\leadsto_{\mu} x$.
		It is easily seen that $\sim_{\mu}$ is an equivalence relation on $X$, and we write $[x]_{\mu}$ for the equivalence class of a point $x$ with respect to $\sim_{\mu}$.
	\end{definition}
	
	Intuitively, the condition $x\leadsto_{\mu} x'$ says that, in terms of distances to a sample from $\mu$, any point sufficiently close to $x$ is indistinguishable from some point close to $x'$.
	For an example where the equivalence relation  $\sim_{\mu}$ has non-trivial equivalence classes, consider $\sone = \R/\Z$ with its geodesic metric $d$, and $\mu = \frac{1}{2}\delta_{0\text{ mod } 1} + \frac{1}{2}\delta_{\frac{1}{2}\text{ mod } 1}$.
	Then we have $(\frac{1}{4} + t)\text{ mod } 1 \sim_{\mu} (\frac{3}{4}-t) \text{ mod } 1$  for any $t\in [-\frac{1}{4},\frac{1}{4}]$.
	
	\begin{theorem}\label{thm:Fp-full-SLLN}
		Suppose  that $(X,d)$ has the Heine-Borel property, and that $\mu\in \mathcal{P}_p(X)$.
		If $F_p(\mu) =[x]_{\mu}$ (respectively, $F_p^\ast(\mu) =[x]_{\mu}$) for some $x\in X$, then $F_p(\bar \mu_n) \to F_p(\mu)$ (respectively, $F_p^\ast(\bar \mu_n) \to F_p^\ast(\mu)$) in $\Fell\vee \uphausdorfftopo$ almost surely.
	\end{theorem}
	
	\begin{proof}
		By Theorem~\ref{thm:Frechet-SLLN-hausdorff}, we only need to consider the Kuratowski lower limit.
		To do this, we return to the notation of the proof of Theorem~\ref{thm:Frechet-SLLN}.
		Let $x'\in F_p(\mu,C)$ be arbitrary, and take any open set $U\subseteq X$ with $x'\in U$.
		Assume for the sake of contradiction that, on some event of positive probability, the sets $F_p(\bar \mu_n,C_n)$ do not eventually intersect $U$.
		Then on this event there exists a subsequence $\{n_k\}_{k\in\N}$ such that $F_p(\bar \mu_{n_k},C_{n_k})$ and $U$ are disjoint for all $k\in\N$.
		Take any sequence $\{x_{n_k}\}_{k\in\N}$ with $x_{n_k}\in F_p(\bar \mu_{n_k},C_{n_k})$, and apply Lemma~\ref{lem:joint-Frechet-cpt} to get that there exists a further subsequence $\{k_j\}_{j\in\N}$ and a point $x\in F_p(\mu,C)$ with $x_{n_{k_j}}\to x$.
		By assumption, $x\sim_{\mu} x'$.
		So, we can get $\varepsilon > 0$ sufficiently small such that $B_{\varepsilon}(x')\subseteq U$, and such that there exists a measurable map $\phi:B_{\varepsilon}(x)\to B_{\varepsilon}(x')$ with $d(\phi(z),\cdot) =d(z,\cdot)$ holding $\mu$-almost everywhere for all $z\in B_{\varepsilon}(x)$.
		For sufficiently large $j\in\N$, it follows that $x_{n_{k_j}}\in B_{\varepsilon}(x)$, hence
		
		\begin{equation}
			\frac{1}{n_{k_j}}\sum_{i=1}^{n_{k_j}}d^p(\phi(x_{n_{k_j}}),Y_i) = \frac{1}{n_{k_j}}\sum_{i=1}^{n_{k_j}}d^p(x_{n_{k_j}},Y_i).
		\end{equation}
		This shows that $\{\phi(x_{n_{k_j}})\}_{j\in\N}$ is eventually in $F_p(\bar \mu_{n_{k_j}},C_{n_{k_j}})\cap U$, and this is a contradiction since $\{n_k\}_{k\in\N}$ was constructed so that $F_p(\bar \mu_{n_{k}},C_{n_k})$ and $U$ were disjoint.
	\end{proof}
	
	The preceding result essentially guarantees that there is full Kuratowski convergence (or equivalently, convergence in the Fell topology) of the empirical Fr\'echet $p$-means (respectively, the empirical support-restricted Fr\'echet  $p$-means) provided that $F_p(\mu)$ (respectively, $F_p^\ast(\mu)$) is a singleton once we pass to a suitable quotient.
	To formalize this claim, let $X_{\mu}$ denote the collection of all $\sim_{\mu}$-equivalence classes of $X$, and let it be endowed with the quotient topology.
	Define $\pi:X\to X_{\mu}$ by $\pi(x) := [x]_{\mu}$, and note that, by the universal property of the quotient topology, there exists a unique continuous map $\tilde{g}_{p,\mu}:X_{\mu}\to [0,\infty)$ such that $\tilde{g}_{p,\mu} \circ \pi= g_{p,\mu}$.
	In this language, the hypothesis of Theorem~\ref{thm:Fp-full-SLLN} is just that $\tilde{g}_{p,\mu}$ has a unique minimizer on $X_{\mu}$ (respectively, $\pi(\supp(\mu))$).
	The following result also holds.
	
	\begin{corollary}\label{cor:Fp-singleton-SLLN}
		Suppose that $(X,d)$ has the Heine-Borel property, and that $\mu\in\mathcal{P}_p(X)$. If $F_p(\mu)$ (respectively, $F_p^\ast(\mu)$) consists of a single point, then  $F_p(\bar \mu_n)\to F_p(\mu)$ (respectively, $F_p^\ast(\bar \mu_n)\to F_p^\ast(\mu)$) in $\Fell\vee\uphausdorfftopo$ almost surely.
	\end{corollary}
	
	\begin{proof}
		By Theorem~\ref{thm:Fp-full-SLLN} and adopting the notation therein, it suffices to show that $F_p(\mu,C) = \{x\}$ implies $[x]_{\mu} = \{x\}$.
		To do this, suppose that $x\sim_{\mu}x'$, and, for sufficiently small $\varepsilon > 0$, get a measurable map $\phi_{\varepsilon}:B_{\varepsilon}(x)\to B_{\varepsilon}(x')$ satisfying the definition of $\sim_{\mu}$.
		Then, the points $\{\phi_{\varepsilon}(x)\}_{\varepsilon > 0}$ have $d(\phi_{\varepsilon}(x),\cdot) = d(x,\cdot)$ holding $\mu$-almost everywhere, as well as $\phi_{\varepsilon}(x)\to x'$ as $\varepsilon \to 0$.
		Therefore, 
		
		\begin{equation}
			\frac{1}{n}\sum_{i=1}^{n}d^p(x,Y_i) = \frac{1}{n}\sum_{i=1}^{n}d^p(\phi_{\varepsilon}(x),Y_i)
		\end{equation}
		for all $n\in\N$ almost surely, so, taking first $\varepsilon \to 0$ and then $n\to\infty$ with the classical SLLN, we get $g_{p,\mu}(x) = g_{p,\mu}(x')$.
		This implies $x'\in F_p(\mu)$ and hence $x'=x$ as claimed.
	\end{proof}
	
	Next we show that, for finite metric spaces, the hypothesis of Theorem~\ref{thm:Fp-full-SLLN} is not only sufficient but also necessary.
	Observe that, if $(X,d)$ is discrete and $\mu\in\mathcal{P}(X)$, then points $x,x'\in X$ have $x\sim_{\mu}x'$ if and only if $d(x,\cdot) = d(x',\cdot)$ holds $\mu$-almost everywhere.
	
	\begin{theorem}\label{thm:full-SLLN-finite}
		Suppose that $(X,d)$ is finite and that $\mu\in \mathcal{P}(X)$. Then, $F_p(\bar \mu_n) \to F_p(\mu)$ (respectively, $F_p^\ast(\bar \mu_n) \to F_p^\ast(\mu)$) in $\Fell$ almost surely if and only if $F_p(\mu) =[x]_{\mu}$ (respectively, $F_p^\ast(\mu) =[x]_{\mu}$) for some $x\in X$. 
	\end{theorem}
	
	\begin{proof}
		By Theorem~\ref{thm:Fp-full-SLLN} we have $F_p(\bar \mu_n,C_n)\to F_p(\mu,C)$ in $\Fell\vee \uphausdorfftopo$ hence in $\Fell$.
		Thus, adopting the notation of Theorem~\ref{thm:Fp-full-SLLN} and again applying Lemma~\ref{lem:kuratowski-hausdorff-relations}, it suffices to show that $F_p(\bar \mu_n,C_n)\to F_p(\mu,C)$ in $\Fell$ implies $F_p(\mu,C) = [x]_{\mu}$.
		Suppose that $F_p(\bar \mu_n,C_n) \to F_p(\mu,C)$ in $\Fell$ holds almost surely. 
		By Lemma~\ref{lem:Fp-nonempty} there exists some $x\in F_p(\mu,C)$. 
		Let $x'\in F_p(\mu,C)$ be arbitrary and set $\xi_i = d^p(x,Y_i)-d^p(x',Y_i)$. Write $S_n = \sum_{i=1}^{n}\xi_i$. 
		Of course, $F_p(\bar \mu_n,C_n) \to F_p(\mu,C)$ in $\Fell$ implies 
		$F_p(\mu,C)\subseteq \kurinner_{n\to\infty}F_p(\bar \mu_n,C_n)$, and, since $(X,d)$ is discrete, this implies that $S_n = 0$ for
		all sufficiently large $n\in\N$ and hence $\xi_n = 0$ for all sufficiently large $n \in \N$.  Because the $\xi_n$ are i.i.d., this
		can only happen if the common distribution of the $\xi_n$ is the point mass at $0$.
		Therefore, we must have $x\sim_{\mu} x'$, hence $[x]_{\mu}\supseteq F_p(\mu,C)$.
		Moreover, if $x'\in X$, and $x\sim_{\mu} x'$, then $g_{p,\mu}(x)=g_{p,\mu}(x')$, so $[x]_{\mu}\subseteq F_p(\mu,C)$.
		This shows $F_p(\mu,C)=[x]_{\mu}$, as claimed.
	\end{proof}
	
	It would be interesting to understand whether the hypothesis of Theorem~\ref{thm:Fp-full-SLLN} is both necessary and sufficient for the conclusion to hold in general Heine-Borel spaces.
	If not, it is possible that some other notion of equivalence (which must agree with $\sim_{\mu}$ for finite metric spaces) is both necessary and sufficient.
	
	We emphasize that identifying the contexts in which the empirical restricted Fr\'echet means converge in a $T_2$ topology is not just a mathematical curiosity; from statistical and computational perspectives, the distinction between these topologies is crucial.
	We outline a heuristic explanation of this phenomenon, focusing only on the unrestricted Fr\'echet $p$-means for the sake of simplicity:
	Suppose that one has access to i.i.d. samples $Y_1,\dots Y_n$, and that the goal is to estimate the set $F_p(\mu)$ from this data.
	A naive approach is to compute $F_p(\bar \mu_n)$ and treat it as an estimate for $F_p(\mu)$.
	If it is known that $F_p(\bar \mu_n)\to F_p(\mu)$ in some $T_2$ topology almost surely, then $F_p(\bar \mu_n)$ indeed provides a good estimate for $F_p(\mu)$ in the limit.
	In general, however, $F_p(\bar \mu_n)$ can be much smaller than $F_p(\mu)$.
	(Consider Example~\ref{ex:discrete-uniform} in the previous section: $F_p(\mu)$ consists of the entire space, but, with high probability, $F_p(\bar \mu_n)$ is a singleton!)
	
	One way to bootstrap the naive estimate to a slightly better estimate is to compute each of the sets $F_p(\bar \mu_1),\ldots F_p(\bar \mu_n)$ and to let the estimate $\hat{F}(Y_1,\dots Y_n)$ consists of all ``accumulation points'' of this sequence of sets.
	It appears that this improved estimator suffers from significant computational drawbacks, so it would be interesting to understand whether some of the computations can be recycled in such a way that makes this procedure reasonable.
	One idea towards this end, at least for finite metric spaces, is to pre-compute the Voronoi diagram $\mathbf{K}_p = \{K_{x}\}_{x\in X}$ and to track the sets which are hit by the sequence of measures $\{\bar \mu_k\}_{k=1}^{n}$.
	(Since the data are in fact exchangeable, an even better estimate could be given by $\hat{F}(Y_1,\dots Y_n) = \bigcup_{\sigma\in S_n}\hat{F}(Y_{\sigma(1)},\ldots Y_{\sigma(n)})$ where $S_n$ represents the set of all permutations of $\{1,\dots n\}$.
	Of course, the enormity of this computational effort suggests that some clever work would be required in devising an algorithm that makes this procedure reasonable.)
	
	\subsection{Geometric Applications}
	
	In this final subsection, we explore some applications of this theory for metric and Riemannian geometry.
	By a \textit{global non-positively curved (NPC) space}, we mean a complete metric space $(X,d)$ such that for all $x_0, x_1 \in X$ there exists a point $y \in X$ with the property that for all points $z \in X$
	\[
	d^2(z,y) \le \frac{1}{2} d^2(z,x_0) + \frac{1}{2} d^2(z,x_1) - \frac{1}{4} d^2(x_0, x_1).
	\]
	(Some authors refer to these as \textit{Hadamard spaces} or \textit{CAT(0) spaces}.)
	Examples include complete, simply connected Riemannian manifolds of non-positive sectional curvature, metric trees, Hilbert spaces, and Euclidean Bruhat--Tits buildings.
	
	It is known \cite[Section~4]{Sturm_03} that if $(X,d)$ is a global NPC space then $F_2(\mu)$ is a singleton for all $\mu \in \mathcal{P}_2(X)$.
	Hence, we can write $b(\mu)$ for the unique element of $F_2(\mu)$; the map $b:\mathcal{P}_2(X)\to X$ is called the \textit{$d^2$-barycenter map}.
	In \cite[Section~6, and Proposition 6.6]{Sturm_03} it is proved that $b(\bar \mu_n)\to b(\mu)$ almost surely as $n\to\infty$ whenever $(X,d)$ is a global NPC space and $\mu\in \mathcal{P}_{\infty}(X)$.
	Our work leads us to a similar result.
	
	\begin{theorem}\label{thm:NPC-SLLN}
		Suppose that $(X,d)$ is a global NPC space with the Heine-Borel property and $\mu \in \mathcal{P}_2(X)$.
		Then, $b(\bar \mu_n)\to b(\mu)$ almost surely as $n\to\infty$.
	\end{theorem}
	
	\begin{proof}
		Immediate from Theorem~\ref{thm:Fp-full-SLLN}.
	\end{proof}
	
	Note that neither set of the hypotheses in \cite[Proposition 6.6]{Sturm_03} or in our Theorem~\ref{thm:NPC-SLLN} implies the other.
	However, our work generalizes nicely to a setting which allows for some small positive curvature:
	Let $(X,d)$ be the metric corresponding to a complete Riemannian manifold $M$.
	Suppose that the sectional curvature of $M$ is upper bounded by $\Delta > 0$, and write $R > 0$ for the injectivity radius of $M$.
	Then define
	\begin{equation}
		\rho(p,\Delta,R) = \begin{cases}
			\frac{1}{2}\min\left\{R,\frac{\pi}{2\sqrt{\Delta}}\right\}&\text{ if }1 \le p \le 2 \\
			\frac{1}{2}\min\left\{R,\frac{\pi}{\sqrt{\Delta}}\right\} &\text{ if } p\ge 2 \\
		\end{cases}.
	\end{equation}
	By the main result of \cite{Afsari}, if we have $1 < p < \infty$ and that $\mu\in \mathcal{P}_{\infty}(X)$ has support contained in a ball of radius $\rho < \rho(p,\Delta,R)$, then it follows that $F_p(\mu)$ is a singleton; write $b(\mu)\in X$ for this element.
	Moreover, $\bar \mu_n$ has support contained in the same ball for all $n\in\N$ almost surely, so we can write $b(\bar \mu_n)\in X$ for $n\in\N$ for these (random) elements.
	Then we have the following.
	
	\begin{theorem}
		Let $(X,d)$ be the metric space arising from a complete Riemannian manifold whose sectional curvature is upper bounded by $\Delta > 0$ and whose injectivity radius is $R > 0$, and fix $1 < p < \infty$.
		If $\mu\in\mathcal{P}_{\infty}(X)$ has support contained in a ball of radius $\rho < \rho(p,\Delta,R)$, then $b(\bar \mu_n)\to b(\mu)$ almost surely as $n\to\infty$.
	\end{theorem}
	
	\begin{proof}
		Immediate from \cite{Afsari} and Theorem~\ref{thm:Fp-full-SLLN}.
	\end{proof}
	
	\nocite{*}
	\bibliography{FrechetMeanSLLN}
	\bibliographystyle{plain}
	
\end{document}